\documentclass[11pt,letterpaper]{amsart}
\usepackage{mathtools}
\usepackage{mathdots}
\usepackage{amssymb}
\usepackage{mathrsfs}
\usepackage[pagebackref,colorlinks=true,linkcolor=blue,citecolor=blue]{hyperref}
\usepackage{tikz-cd}

\tikzset{
  symbol/.style={
    draw=none,
    every to/.append style={
      edge node={node [sloped, allow upside down, auto=false]{$#1$}}}
  }
}

\newcommand{\Z}{\mathbb{Z}}

\newcommand{\R}{\mathbb{R}}
\newcommand{\BC}{\mathbb{C}}

\newcommand{\std}{{\mathrm{std}}}

\newcommand{\SL}{\mathrm{SL}}
\newcommand{\GL}{\mathrm{GL}}
\newcommand{\SO}{\mathrm{SO}}
\newcommand{\RO}{\mathrm{O}}
\newcommand{\Sp}{\mathrm{Sp}}

\newcommand{\RG}{\mathrm{G}}

\newcommand{\ur}{\mathrm{ur}}

\newcommand{\mfr}[1]{\mathfrak{#1}}

\newcommand{\MSeg}{\underline{\mathrm{MSeg}}}

\newcommand{\half}[1]{\frac{#1}{2}}

\newcommand{\comment}[1]{}

\newcommand{\AZ}{\mathsf{AZ}}
\newcommand{\ABV}{\mathrm{ABV}}

\newtheorem{thm}{Theorem}[section]
\newtheorem{cor}[thm]{Corollary}
\newtheorem{lemma}[thm]{Lemma}
\newtheorem{prop}[thm]{Proposition}
\newtheorem {conj}[thm]{Conjecture}

\newtheorem {ques/conj}[thm]{Question/Conjecture}

\newtheorem{defn}[thm]{Definition}
\newtheorem{remark}[thm]{Remark}

\newtheorem{exmp}[thm]{Example}

\newtheorem{algo}[thm]{Algorithm}

\newtheorem*{globalcond*}{Global Condition}
\newtheorem*{localcond*}{Local Condition}

\newtheorem*{globalconj*}{Global Conjecture}
\newtheorem*{localconj*}{Local Conjecture}

\newtheorem*{nonzero*}{Conjecture on the non-vanishing of the normalized intertwining operators}
\newtheorem*{holo*}{Conjecture on the holomorphicity of the normalized intertwining operators}

\DeclareMathOperator{\Ad}{Ad}

\DeclareMathOperator{\Aut}{Aut}
\DeclareMathOperator{\Hom}{Hom}

\DeclareMathOperator{\End}{End}

\DeclareMathOperator{\id}{id}

\numberwithin{equation}{section}

\let\oldbullet\bullet
\renewcommand{\bullet}{{\vcenter{\hbox{\tiny$\oldbullet$}}}}

\begin{document}

\title[Algorithms on computing the Pyasetskii involution]{Algorithms on the Pyasetskii involution on local Langlands parameters of classical groups}

\author{Alexander Hazeltine}
\address{Department of Mathematics\\
University of Michigan\\
Ann Arbor, MI, 48109, USA}
\email{ahazelti@umich.edu}

\author{Chi-Heng Lo}
\address{Department of Mathematics\\
National University of Singapore\\
119076, Singapore}
\email{{ch\_lo@nus.edu.sg}}


\date{\today}


\begin{abstract}
 We give an algorithm to compute the Pyasetskii involution for $\Sp_{2n}$, $\SO_{2n+1}$ and $\RO_{2n}$. The algorithm is a combination of M{\oe}glin-Waldspurger's algorithm for the Pyasetskii involution for $\GL_n$ (\cite{MW86}) and Lanard-M{\'i}nguez's algorithm for the Aubert-Zelevinsky involution of bad parity representations for classical groups (\cite{LM25}). In particular, we give a geometric interpretation of the bad parity case of Lanard-M{\'i}nguez's algorithm.
\end{abstract}

\maketitle

\section{Introduction}

Let $F$ be a non-Archimedean local field of characteristic zero with Weil group $W_F$. Let $G$ be a connected reductive algebraic group defined and quasi-split over $F$. For simplicity, we do not distinguish between $G$ and its group of $F$-points. Denote by $G^{\vee}$ the complex dual group of $G$, and by ${}^L G$ its Langlands dual group.  
Let $\Pi(G)$ be the set of irreducible smooth representations of $G$, and $\Phi(G)$ the set of $L$-parameters for $G$, i.e.~the set of $G^{\vee}$-conjugacy classes of admissible homomorphisms
\[
\phi: W_F \times \mathrm{SL}_2(\BC) \longrightarrow {}^L G.
\]
To each $L$-parameter $\phi$, we attach an infinitesimal parameter (an $L$-parameter which is trivial on $\textrm{SL}_2(\BC)$)
\[
\lambda_\phi(w) := \phi\!\left(w, 
\begin{pmatrix}
|w|^{1/2} & \\ & |w|^{-1/2}
\end{pmatrix}
\right), \qquad w \in W_F.
\]
We let $G_n$ denote the split $\SO_{2n+1}(F)$, $\Sp_{2n}(F)$ or quasi-split $\RO_{2n}(F)$ and refer to them as classical groups throughout this paper. When $G_n=\RO_{2n}(F)$, we identify its complex dual group as $\RO_{2n}(\BC)$. Two admissible homomorphisms $\phi_1, \phi_2: W_F \times \SL_2(\BC) \to \RO_{2n}(\BC)$ are equivalent if they are conjugate to each other under $\RO_{2n}(\BC)$.

Consider the set $\Phi_{\lambda}(G) := \{\phi \in \Phi(G) \mid \lambda_{\phi} = \lambda\}$, which is finite and naturally in bijection with the set of orbits $V_{\lambda}/H_{\lambda}$. Here, $V_{\lambda}$ denotes the Vogan variety associated to $\lambda$, and $H_{\lambda}$ is the centralizer of $\lambda$, which acts on $V_{\lambda}$ by conjugation. See \S \ref{sec Vogan variety} for details.
The geometric structure of $V_{\lambda}/H_{\lambda}$ endows $\Phi_{\lambda}(G)$ with two key features: a partial order $\geq_C$ given by the closure relation, and an involution $\phi \mapsto \widehat{\phi}$, known as the \emph{Pyasetskii involution}. Both the partial order and the involution play crucial roles in the study of local Arthur packets and ABV-packets. Motivated by their significance, it becomes important to provide explicit algorithms for their computation, particularly in settings where $\Phi_{\lambda}(G)$ admits a clear combinatorial description. In this paper, we focus on the cases when $G$ is either a general linear group or a classical group.

For general linear groups, the local Langlands correspondence gives a bijection between $\Pi(\GL_n(F))$ and $\Phi(\GL_n(F))$ (\cite{Hen00, HT01, Sch13}). In particular, we identify $d$-dimensional irreducible representations of $W_F$ with supercuspidal representations of $\GL_d(F)$. Under this identification, the Zelevinsky classification (\cite[Theorem 6.1]{Zel80}) provides a parametrization of $\Phi(\GL_n(F))$ by multi-segments (see Definition \ref{def segments}). Zelevinsky showed in \cite[Theorem 2.2]{Zel81} that the closure ordering on $\Phi_{\lambda}(\GL_n(F))$ corresponds to the partial ordering on multi-segments defined by the union-intersection relation. In the same work, he conjectured that the Pyasetskii involution for $L$-parameters coincides with the Zelevinsky involution (\cite[\S 9]{Zel80}) for the associated irreducible representations. This conjecture was subsequently proven by M{\oe}glin and Waldspurger in \cite{MW86}, who also provided a combinatorial algorithm on multi-segments that realizes these involutions, now known as the M{\oe}glin-Waldspurger algorithm. Subsequently, Knight and Zelevinsky, in \cite{KZ96}, also established an explicit closed formula for computing Pyasetskii involution, using the theory of flows in network.

For the classical groups $G_n$ under consideration, there exists a standard embedding $\std: {}^L G_n \to \GL_N(F)$ (where $N=2n+1$ if $G_n = \Sp_{2n}(F)$, and $N=2n$ otherwise), which induces an injection on $L$-parameters via $\std_{\ast}: \Phi(G_n) \to \Phi(\GL_N(F))$ by composition (see \cite[Theorem 8.1]{GGP12}). In particular, we may view $\Phi(G_n)$ as a subset of $\Phi(\GL_N(F))$ and parametrize its elements in terms of multi-segments, just as for general linear groups. It is known that for any $\phi_1, \phi_2 \in \Phi_{\lambda}(G)$, we have $\phi_1 \geq_C \phi_2$ if and only if $\std \circ \phi_1 \geq_C \std \circ \phi_2$.
Moreover, the closure order can be described via the \emph{rank triangle} introduced in \cite[\S 3]{CR22}; for technical convenience, we rephrase this as the language of \emph{rank matrices} in \S \ref{sec rank matrix}, highlighting their additivity properties. Note that for $\SO_{2n}(F)$, the map $\std_{\ast}$ is not injective and the criterion for closure ordering does not apply, which is why we instead treat $\RO_{2n}(F)$.

While the structure of the closure ordering in this context is well understood, an explicit algorithm for computing the Pyasetskii involution for these classical groups has not previously appeared in the literature. The main result of this paper is to provide such an algorithm.

\begin{thm}\label{thm main intro}
    Let $\phi$ be an $L$-parameter of $G_n$ that corresponds to a multi-segment $\mfr{m}$. We decompose $\mfr{m}$ as (see Definition \ref{def rho part})
    \begin{align}\label{eq decomp intro}
        \mfr{m}=\bigsqcup_{[\rho]\in R} \mfr{m}_{[\rho]}.
    \end{align}
    Then the Pyasetskii involution $\widehat{\phi}$ is the $L$-parameter corresponding to the multi-segment
    \[ \mfr{m}^{\sharp}= \bigsqcup_{[\rho]\in R} \mfr{m}_{[\rho]}^{\sharp},\]
    where $\mfr{m}_{[\rho]}^{\sharp}$ is given by 
    \begin{itemize}
        \item the M{\oe}glin-Waldspurger algorithm if $[\rho]$ is of good parity or non-selfdual,
        \item the Lanard-M{\'i}nguez algorithm for the Aubert-Zelevinsky involution in the bad parity case (see \cite[\S 5.2]{LM25} or Algorithm \ref{algo LM25}), if $[\rho]$ is of bad parity.
    \end{itemize}
    We refer to \S \ref{sec decomp} for the definition of non-selfdual and to Definition \ref{def gp bp} for good and bad parity.
\end{thm}
In particular, this gives a geometric interpretation of the Lanard-M{\'i}nguez algorithm in the bad parity case.

We briefly outline our approach. We first observe that both the Vogan variety $V_{\lambda}$ and the centralizer group $H_{\lambda}$ naturally decompose as products (see \eqref{eq decomp rho classical}), which induces the decomposition in \eqref{eq decomp intro}. As a result, it is enough to compute $\mfr{m}_{[\rho]}^{\sharp}$ for each $[\rho]$ separately (see Lemma \ref{lem phihat}). When $[\rho]$ is non-selfdual, the corresponding Vogan variety is naturally isomorphic to that of a general linear group, so the computation of the Pyasetskii involution reduces to the well-known case done by M{\oe}glin-Waldspurger (see Proposition \ref{prop nsd}).

In the good parity case, the key observation is that the Pyasetskii involution on $\Phi_{\std \circ \lambda}(\GL_N(F))$ preserves the subset corresponding to $\Phi_{\lambda}(G_n)$ (Proposition \ref{prop involution of orbit gp}). We can also compare the closure orderings between $\std \circ \widehat{\phi}$ and $\widehat{\std \circ \phi}$ using rank matrices (Corollary \ref{cor hat GL classical}). Consequently, the desired result follows from a simple yet very effective lemma (see Lemma \ref{lem involution poset}).

\begin{lemma}\label{lem involution poset intro}
Let $(S, \geq)$ be a finite partially ordered set. Suppose there exist two involutions $\iota_1, \iota_2$ on $S$ such that $\iota_1(s) \geq \iota_2(s)$ for all $s \in S$. Then $\iota_1 = \iota_2$.
\end{lemma}

The remaining case is the bad parity case which turns out to be the most intricate. To address this, we adapt the proof of \cite[Proposition II.6]{MW86} from the setting of $\GL_n(F)$ to that of the classical groups $G_n^{\vee}$. Since both the Pyasetskii involution and the Lanard-M{\'i}nguez algorithm are known to be involutions, thanks to Lemma \ref{lem involution poset intro} again, we only need to modify (the simpler) half of the original argument in \cite[Proposition II.6]{MW86}. This modification is done in \S \ref{sec bad}.

Finally, we prove in Proposition \ref{prop ABV AZ bad} that the bad parity case of Theorem \ref{thm main intro} is indeed a consequence of the following expected property of ABV-packets: For any $L$-parameter of $G$, we have
\begin{align}\label{eq ABV AZ intro}
    \Pi_{\widehat{\phi}}^{\ABV}(G)= \{ \widehat{\pi}\ | \ \pi \in \Pi_{\phi}^{\ABV}\}.
\end{align}
See \S \ref{sec AZ} for background about ABV-packets. Since we have proved Theorem \ref{thm main intro} unconditionally, it provides an evidence for \eqref{eq ABV AZ intro}.

Here is an outline of the paper.
We begin by recalling in \S\ref{sec Vogan variety} the definition of the Pyasetskii involution through the viewpoint of Vogan varieties, along with various preliminaries. In \S\ref{sec Reduction of Pyasetskii involution for classical groups}, we provide several reductions in the computation of the Pyasetskii involution. Next, in \S\ref{sec A lemma on involutions on finite partially ordered sets}, we prove an elementary lemma on finite partially ordered sets which allows us to avoid having to generalize certain arguments in \cite[Proposition II.6]{MW86}. In \S\ref{sec AZ}, we recall the representation theoretic analogue of the Pyasetskii involution, the Aubert-Zelevinsky involution. Then we prove the computation of the Pyasetskii involution in the good parity case in \S\ref{sec The good parity case}. We complete the computation of the Pyasetskii involution by giving an algorithm for the bad parity case in \S\ref{sec bad}. Lastly, we discuss how the computation in the bad parity case is evidence for the conjectural Equation \eqref{eq ABV AZ intro} in \S\ref{sec ABV}.

\subsection*{Acknowledgements} 
The authors would like to thank Baiying Liu and Qing Zhang for the constant support and encouragement.

\section{Vogan variety}\label{sec Vogan variety}

Throughout this section, we let $\RG$ be a connected reductive group defined and quasi-split over a non-Archimedean local field $F$ of characteristic 0, although we later allow for $\RG$ to be an even orthogonal group which is defined and quasi-split over $F$ (see \S\ref{sec even orthog}). Set $G = \RG(F)$, and recall that the set of $L$-parameters for $G$ is denoted by $\Phi(G)$.

For any $\phi \in \Phi(G)$, there is an associated homomorphism
\[
\lambda_{\phi} : W_F \to \RG^\vee(\mathbb{C})
\]
given by
\[
\lambda_{\phi}(w) := \phi\left(w, \begin{pmatrix}
|w|^{\frac{1}{2}} & \\ & |w|^{-\frac{1}{2}}
\end{pmatrix}\right),
\]
which is called an \emph{infinitesimal parameter} of $G$; that is, a continuous map whose image consists of semisimple elements (see \cite[Section 4.1]{CFMMX22}). Conversely, for a given infinitesimal parameter $\lambda$, we let
\[
\Phi_{\lambda} (G):= \{ \phi \in \Phi(G) \mid \lambda_{\phi} = \lambda \}.
\]
Following \cite{Vog93}, we consider the \emph{Vogan variety} $V_{\lambda}$ for each infinitesimal parameter $\lambda$
\[ V_{\lambda}:= \{ x \in \mathfrak{g}^\vee \ | \ \Ad(\lambda(w))x= |w|x, \forall w \in W_F\}, \]
where $\mathfrak{g}^\vee$ is the Lie algebra of $\RG^\vee(\mathbb{C})$. The vector space $V_{\lambda}$ admits an action of the centralizer group
\[ H_{\lambda}:= \{g \in \RG^\vee(\BC) \ | \ \lambda(w)g= g\lambda(w), \forall w \in W_F \}\]
with finitely many orbits (see \cite[Proposition 5.6]{CFMMX22}),
and for each $\phi \in \Phi_{\lambda}(G)$, the element
\[ X_{\phi}:= d (\phi|_{\SL_2(\BC)}) \left(  \begin{pmatrix}
0&1\\0&0
\end{pmatrix}\right)\]
is in $V_{\lambda}$. Let $C_{\phi}$ denote the $H_{\lambda}$-orbit of $X_{\phi}$. Then we obtain a map
\begin{align*}
     \Phi_{\lambda}(G)& \to V_{\lambda}/H_{\lambda},\\
     \phi& \mapsto C_{\phi},
\end{align*}
which is in fact a bijection (see \cite[Proposition 4.2]{CFMMX22}).

\subsection{Even Orthogonal Groups}\label{sec even orthog}

The results in \cite{CFMMX22} are stated for connected reductive groups. For our purposes, we require analogous results for even orthogonal groups which we explicate in this subsection.

Let $\mathrm{G}_n=\mathrm{O}_{2n}$ denote an even orthogonal group which is defined and quasi-split over $F$. Fix an element $c\in \RG_n\setminus\mathrm{G}_n^0$ defined over $F$ with $c^{-1}=c$, where $\mathrm{G}_n^0=\SO_{2n}$ is the connected component of the identity, and consider the outer automorphism $\theta_c:\SO_{2n}\rightarrow\SO_{2n}$ defined by $\theta_c(g)=cgc^{-1}$ for any $g\in\SO_{2n}.$ Since $\theta_c$ preserves a pinning, there exists a dual automorphism $\widehat{\theta}_c:\SO^\vee_{2n}(\BC)\rightarrow{\SO}^\vee_{2n}(\BC).$ We set $\mathrm{O}^\vee_{2n}(\BC):=\SO^\vee_{2n}(\BC)\rtimes\langle\widehat{\theta}_c\rangle\cong\mathrm{O}_{2n}(\BC).$ 

From the definition, we have that the Lie algebra $\mathfrak{o}_{2n}$ of $\mathrm{O}_{2n}$ is the Lie algebra of $\SO_{2n}.$ We define the set of $L$-parameters of $\mathrm{O}_{2n}$, denoted $\Phi(\mathrm{O}_{2n}),$ to be the set of $L$-parameters of $\SO_{2n}.$ However, we say that $\phi_1,\phi_2\in\Phi(\mathrm{O}_{2n})$ are equivalent if they are conjugate in $\mathrm{O}^\vee_{2n}(\BC).$ Similarly, an infinitesimal parameter of $\mathrm{O}_{2n}$ is an infinitesimal parameter of $\SO_{2n}$ (with equivalence again taken to be conjugation in $\mathrm{O}^\vee_{2n}(\BC)$). We note that since $c$ is defined over $F,$ the Galois group acts trivially on $\langle \theta_c \rangle$.

With the above setup, the definitions in the previous section generalize to $\mathrm{O}_{2n}.$ That is, the \emph{Vogan variety} $V_{\lambda}$ for each infinitesimal parameter $\lambda$ of $\mathrm{O}_{2n}$ is given by ($\mathfrak{o}_{2n}=\mathfrak{so}_{2n}$)
\[ V_{\lambda}:= \{ x \in \mathfrak{o}^\vee_{2n} \ | \ \Ad(\lambda(w))x= |w|x, \forall w \in W_F\}. \]
Again, $V_{\lambda}$ admits an action of the group
\[ H_{\lambda}:= \{g \in \mathrm{O}^\vee_{2n}(\BC) \ | \ \lambda(w)g= g\lambda(w), \forall w \in W_F \}.\]

We let $H_\lambda^\mathrm{SO}=\{g \in \mathrm{SO}^\vee_{2n}(\BC) \ | \ \lambda(w)g= g\lambda(w), \forall w \in W_F \}.$  Since $V_\lambda / H_\lambda^\mathrm{SO}$ has finitely many orbits (see \cite[Proposition 5.6]{CFMMX22}), so does $V_\lambda/H_\lambda.$

Again, for each $\phi \in \Phi_{\lambda}(\mathrm{O}_{2n})=\{\phi\in\Phi(\mathrm{O}_{2n}) \ | \ \lambda_\phi=\lambda\}$, the element
\[ X_{\phi}:= d (\phi|_{\SL_2(\BC)}) \left(  \begin{pmatrix}
0&1\\0&0
\end{pmatrix}\right)\]
is in $V_{\lambda}$. Let $C_{\phi}$ denote the $H_{\lambda}$-orbit of $X_{\phi}$. 

\begin{lemma}\label{lemma even orthog L-par orbit bijection}
We have a bijection
\begin{align*}
    \Phi_{\lambda}(\mathrm{O}_{2n})& \to V_{\lambda}/H_{\lambda},\\
     \phi& \mapsto C_{\phi}.
\end{align*}
It is important to note that we regard $\Phi_{\lambda}(\mathrm{O}_{2n})$ as a set modulo the above equivalence here.
\end{lemma}

\begin{proof}

From the connected case (\cite[Proposition 5.6]{CFMMX22}), we have that the map is surjective. Thus it suffices to show that if $C_\phi=C_{\phi'}$ for some $\phi,\phi'\in \Phi_{\lambda}(\mathrm{O}_{2n}),$ then $\phi$ and $\phi'$ are equivalent as $L$-parameters of $\mathrm{O}_{2n}.$ Suppose then that $C_\phi=C_{\phi'}.$ 

If $hX_\phi\in C_{\phi'}$ for some $h\in H_\lambda^\mathrm{SO}$, then the $H_\lambda^\mathrm{SO}$-orbits of $X_\phi$ and $X_{\phi'}$ coincide. By the connected case (\cite[Proposition 5.6]{CFMMX22}), it follows that $\phi$ and $\phi'$ are equivalent as $L$-parameters of $\mathrm{SO}_{2n}$ and therefore, they are also equivalent as $L$-parameters of $\mathrm{O}_{2n}$. This verifies the desired conclusion in this case.

Suppose then that $hX_\phi\in C_{\phi'}$ for some $h\in H_\lambda\setminus H_\lambda^\mathrm{SO}.$ Then $h=g\rtimes \widehat{\theta}_c$ for some $g\in\SO^\vee_{2n}(\BC).$ Since $h\in H_\lambda$, we have $g\widehat{\theta}_c(\lambda(w))g^{-1}=\lambda(w)$
for any $w\in W_F.$ We obtain that $\mathrm{Int}(g)\circ\widehat{\theta}_c\circ\phi\in \Phi_{\lambda}(\mathrm{O}_{2n})$, where $\mathrm{Int}(g)$ denotes the inner conjugation. From the connected case (\cite[Proposition 5.6]{CFMMX22}), it follows that $\phi'$ and $\mathrm{Int}(g)\circ\widehat{\theta}_c\circ\phi$ are equivalent as $L$-parameters of $\mathrm{SO}_{2n}.$ Therefore, we have that $\phi$ and $\phi'$ are equivalent as $L$-parameters of $\mathrm{O}_{2n}.$
\end{proof}



\subsection{Pyasetskii involution} Let $\RG$ be a connected reductive group (or even orthogonal group) that is defined and quasi-split over $F$ and $G=\RG(F).$
Let $\lambda$ be an infinitesimal parameter of $G$. Consider the dual of the Vogan variety
\begin{align*}
     V_{\lambda}^{\ast}:=\{ x \in \mathfrak{g}^\vee \ | \ \Ad(\lambda(w)) x= |w|^{-1} x \}
\end{align*}
and the conormal variety $\Lambda_{\lambda}$ and its subsets $\Lambda_{C}$ (resp. $\Lambda_{C^{\ast}}^{\ast}$) for each $C\in V_{\lambda}/H_{\lambda}$ (resp. $C^{\ast} \in V_{\lambda}^{\ast}/ H_{\lambda}$)
\begin{align*}
    \Lambda_{\lambda}&:= \{ (v, v^{\ast}) \in V_{\lambda} \times V_{\lambda}^{\ast} \ | \ [v, v^{\ast}]=0 \},\\
     \Lambda_{C}&:= \{ (v, v^{\ast}) \in C \times V_{\lambda}^{\ast} \ | \ [v, v^{\ast}]=0 \},\\
      \Lambda_{C^{\ast}}^{\ast}&:= \{ (v, v^{\ast}) \in V_{\lambda} \times C^{\ast} \ | \ [v, v^{\ast}]=0 \},
\end{align*}
where $[\cdot,\cdot]$ denotes the Lie bracket on $\mathfrak{g}^\vee$. Fix a Killing form $B$ on $\mathfrak{g}$; this induces an isomorphism between $V_{\lambda}$ and $V_{\lambda}^{\ast}$. Consequently, we obtain a bijection between $V_{\lambda}/H_{\lambda}$ and $V_{\lambda}^{\ast}/H_{\lambda}$, which we denote by the transpose map $C \mapsto {}^t C$, with inverse $C^{\ast} \mapsto {}^t C^{\ast}$.
The following lemma defines the \emph{Pyasetskii involution} on $ V_{\lambda}/H_{\lambda}$.

\begin{lemma}[{\cite[Lemma 6.5]{CFMMX22}}]\label{lemma Pyasetskii involution}
    For each $H_{\lambda}$-orbit $C$ of $V_{\lambda}$, there exists a unique $H_{\lambda}$-orbit $\widehat{C}$ in $V_{\lambda}$ such that
    \[\overline{\Lambda_{C}}= \overline{\Lambda_{{}^t\widehat{C}}^{\ast}}.\]
\end{lemma}

\begin{remark}\label{rmk even orthog Pyasetskii dual}
    While \cite[Lemma 6.5]{CFMMX22} is stated for connected reductive groups, its extension to even orthogonal groups is straightforward. Indeed, let $\lambda$ be an infinitesimal parameter of $\mathrm{O}_{2n}$ and $\phi\in\Phi_{\lambda}(\mathrm{O}_{2n}).$ Using the notation of \S\ref{sec even orthog}, let $C_\phi$ be the $H_\lambda$-orbit of $X_\phi$ and $C_\phi^0$ be the $H_\lambda^\mathrm{SO}$-orbit of $X_\phi$. By Lemma \ref{lemma even orthog L-par orbit bijection}, we have $C_\phi=C_\phi^0\cup C_{{}^{\widehat{\theta}_c}\phi}^0$, where $C_{{}^{\widehat{\theta}_c}\phi}$ is either empty or denotes the $H_\lambda^\mathrm{SO}$-orbit of $X_{\mathrm{Int}(g)\circ\widehat{\theta}_c\circ\phi}$ if $\mathrm{Int}(g)\circ\widehat{\theta}_c\circ\phi\in \Phi_{\lambda}(\mathrm{O}_{2n})$ for some $g\in\SO^\vee_{2n}(\BC).$ We have that
    $\Lambda_{C_\phi}=\Lambda_{C_\phi^0}\cup \Lambda_{C_{{}^{\widehat{\theta}_c}\phi}^0}.$ It follows from Lemma \ref{lemma Pyasetskii involution} for $\SO_{2n}$ that 
    \[
    \overline{\Lambda_{C_\phi}}=\overline{\Lambda_{C_\phi^0}}\cup \overline{\Lambda_{C_{{}^{\widehat{\theta}_c}\phi}^0}}=\overline{\Lambda_{{}^t\widehat{C_\phi^0}}}\cup \overline{\Lambda_{{}^t\widehat{C_{{}^{\widehat{\theta}_c}\phi}^0}}}=\overline{\Lambda_{{}^t\widehat{C_\phi}}},
    \]
    where $\widehat{C_\phi}=\widehat{C_\phi^0}\cup\widehat{C_{{}^{\widehat{\theta}_c}\phi}^0}$.
    To justify the last equality, we note that $C_\phi^0\neq C_{{}^{\widehat{\theta}_c}\phi}^0$ if and only if ${}^t\widehat{C_\phi^0}\neq {}^t\widehat{C_{{}^{\widehat{\theta}_c}\phi}^0}.$
\end{remark}

\begin{remark}\label{rmk Pyasetskii involution}
    Fix a $f \in V_{\lambda}$ and define $C(f):=\{g \in V_{\lambda}^{\ast}\ | \ [f,g]=0\}$. Then $ \widehat{C} \cap C(f)$ is dense in $C(f)$.
\end{remark}

  For each $L$-parameter $\phi$ in $ \Phi_{\lambda}(G)$, define $\widehat{\phi} \in \Phi_{\lambda}(G)$ such that 
\[\widehat{C_{\phi}}= C_{\widehat{\phi}}.\]
We shall call $\widehat{\phi}$ (resp. $\widehat{C}$) the \emph{Pyasetskii involution} of $\phi$ (resp. $C$).

\subsection{\texorpdfstring{Closure ordering and rank matrix for general linear groups}{}}\label{sec closure ordering}\label{sec rank matrix}
Another important structure of $\Phi_{\lambda}(G)$ is the closure ordering inherited from the geometry of $V_{\lambda}/H_{\lambda}$.
\begin{defn}
For each infinitesimal parameter $\lambda$ of $G$, the finite set $\Phi_{\lambda}(G)$ is equipped with a partial ordering $\geq_C$ defined by $\phi_{1} \geq_C \phi_2$ if $\overline{C_{\phi_1}} \supseteq C_{\phi_2}$.
\end{defn}
In this subsection, we give a sufficient criterion to determine whether $\phi_{1} \geq_{C} \phi_2$ using \emph{rank matrix} (called \emph{rank triangles} in \cite[\S 10.2.1]{CFMMX22}). When $G=G_n$ is the classical groups we consider, this criterion is also necessary. Note that $\SO_{2n}$ is not one of our $G_n$, but $\RO_{2n}$ is.

We begin by considering the case $G=\GL_n(F)$. An infinitesimal parameter $\lambda$ of $\GL_n(F)$ is an $n$-dimensional representation of $W_F$, which can be expressed as
\begin{align}\label{eq inf char 1}
    \lambda = \bigoplus_{i\in I} \rho_i,
\end{align}
where each $\rho_i$ is a $d_i$-dimensional irreducible representation of $W_F$, with $\sum_{i\in I} d_i=n$. Let $W$ denote the underlying vector space of $\lambda$. This decomposes as $W = \bigoplus_{i\in I} W_i$, where $W_F$ acts on $W_i$ via $\rho_i$.

Consider now any $f \in V_{\lambda}$ (resp. $h \in H_{\lambda}$), which we regard as endomorphisms of $W$, so that $f(W_i)$ (resp. $h(W_i)$) are subspaces of $W$. By the definition of $f \in V_{\lambda}$, we see that either $f(W_i)=0,$ or the action of $W_F$ on $f(W_{i})$ is isomorphic to $\rho_i\lvert \cdot \rvert^1$. Similarly, for $h \in H_{\lambda} \subseteq \Aut(W)$, the action of $W_F$ on $h(W_i)$ is isomorphic to $\rho_i$. Equivalently, we have
\begin{align}\label{eq observation decomp Vogan}
    f(W_i) \subseteq \bigoplus_{j \in I,\ \rho_j \cong \rho_i\lvert\cdot\rvert^1} W_j, \quad h(W_i) \subseteq \bigoplus_{j \in I,\ \rho_j \cong \rho_i} W_j.
\end{align}
Thus, $V_{\lambda}$ and $H_{\lambda}$ admit natural decompositions, which we clarify below.

Let $\mathcal{C}$ denote the set of isomorphism classes of finite-dimensional irreducible representations of $W_F$. Define an equivalence relation on $\mathcal{C}$ by $\rho_1 \sim \rho_2$ if and only if there exists an integer $x$ such that $\rho_1 \cong \rho_2\lvert\cdot\rvert^x$. Let $[\rho]$ denote the equivalence class of $\rho$. Using this relation, we can rewrite \eqref{eq inf char 1} by grouping each summand according to its equivalence class. For each class $[\rho]$, define $I_{[\rho]} := \{i \in I\ | \  \rho_i \in [\rho]\}$ and set
\[
\lambda_{[\rho]} := \bigoplus_{i \in I_{[\rho]}} \rho_i,
\]
which is an infinitesimal parameter of some general linear group. We warn the reader that later we will define $\lambda_{[\rho]}$ for infinitesimal parameter $\lambda$ of classical groups with a different definition, so that $\lambda_{[\rho]}$ is still an infinitesimal parameter of some classical group. 

There is a finite set $R$ of equivalence classes such that $\lambda = \bigoplus_{[\rho] \in R} \lambda_{[\rho]}$. The observation \eqref{eq observation decomp Vogan} then tells us that
\begin{align}\label{eq decomp rho GL}
    V_{\lambda} \cong \bigtimes_{[\rho]\in R} V_{\lambda_{[\rho]}},\qquad H_{\lambda} \cong \bigtimes_{[\rho]\in R} H_{\lambda_{[\rho]}}.
\end{align}

Given any $\rho$, we can uniquely write $\rho \cong \rho^u \otimes \lvert\cdot\rvert^x$ for a unitary irreducible representation $\rho^u$ of $W_F$ and some $x \in \R$. Observe that for any $\rho', \rho \in [\rho]$, we have $(\rho')^u = \rho^u$. Accordingly, for $i \in I_{[\rho]}$, define a real number $x_i$ by $\rho_i \cong \rho^u \lvert\cdot\rvert^{x_i}$; in this case, $x_i - x_j \in \Z$ for any $i, j \in I_{[\rho]}$.
Set
\[
e_{[\rho]}^{\min} = \min_{i \in I_{[\rho]}}\{x_i\}, \qquad e_{[\rho]}^{\max} = \max_{i \in I_{[\rho]}}\{x_i\}, \qquad l_{[\rho]} := e_{[\rho]}^{\max} - e_{[\rho]}^{\min}.
\]
For $t \in \Z$, let $W_{\rho, t}$ denote the (possibly zero) subspace of $W$ on which $W_F$ acts as
\[
\bigoplus_{i \in I_{[\rho]},\ \rho_i \cong \rho^u\lvert\cdot\rvert^t} \rho_i.
\]
By Schur's lemma (and with the convention that $\GL(\BC^0) := \{1\}$), we have
\begin{align*}
    H_{\lambda_{[\rho]}} &\cong \bigtimes_{t = e_{[\rho]}^{\min}}^{e_{[\rho]}^{\max}} \GL(W_t), \\
    V_{\lambda_{[\rho]}} &= \bigtimes_{t = e_{[\rho]}^{\min}}^{e_{[\rho]}^{\max} - 1} \Hom_{\BC}(W_t, W_{t+1}),
\end{align*}
where $H_{\lambda_{[\rho]}}$ acts on $V_{\lambda_{[\rho]}}$ by conjugation.

For any $f_{[\rho]} \in V_{\lambda_{[\rho]}}$, write $f_{[\rho]} = (f_t)_{e_{[\rho]}^{\min} \leq t < e_{[\rho]}^{\max}}$, where $f_t \in \Hom_{\BC}(W_t, W_{t+1})$. Then, we associate an $l_{[\rho]} \times l_{[\rho]}$ upper triangular matrix $r_{[\rho]}(f_{[\rho]}) = (r_{[\rho]}(f_{[\rho]})_{ij})_{ij}$, where
\begin{align}\label{eq rk matrix def}
     r_{[\rho]}(f_{[\rho]})_{ij} :=
     \begin{cases}
        \operatorname{rank}\left( f_{e_{[\rho]}^{\max} - i} \circ f_{e_{[\rho]}^{\max} - i - 1} \circ \cdots \circ f_{e_{[\rho]}^{\max} - j} \right) & \text{if } i \leq j, \\
        0 & \text{otherwise}.
     \end{cases}
\end{align}
This matrix $r_{[\rho]}(f_{[\rho]})$ is invariant under the action of $H_{\lambda_{[\rho]}}$. We formalize this with the following definition.

\begin{defn}
Let $\lambda$ be an infinitesimal parameter of $\GL_n(F)$.
\begin{enumerate}
    \item  Let $C$ be an $H_{\lambda}$-orbit of $V_{\lambda}$. For any $f \in C$, write $f = (f_{[\rho]})_{[\rho] \in R}$ via the decomposition \eqref{eq decomp rho GL}. The \emph{rank matrix} of $C$ with respect to $[\rho]$ is defined by $r_{[\rho]}(C) := r_{[\rho]}(f_{[\rho]})$. We then let $r(C) := \{ r_{[\rho]}(C) \}_{[\rho] \in R}$ denote the collection of these rank matrices. For $\phi \in \Phi_{\lambda}(\GL_n(F))$, we likewise define $r(\phi) := r(C_\phi)$.
    \item Given two orbits $C_1$ and $C_2$, we write $r(C_1) \geq r(C_2)$ if $r_{[\rho]}(C_1)_{ij} \geq r_{[\rho]}(C_2)_{ij}$ for all $[\rho] \in R$ and all $1 \leq i, j \leq l_{[\rho]}$.
\end{enumerate}
\end{defn}

The $H_{\lambda}$-orbits are completely determined by the data of $r(C)$: two orbits coincide if and only if their rank matrices agree, and the closure relation $\overline{C_1} \supseteq C_2$ holds if and only if $r(C_1) \geq r(C_2)$. For further details and an explicit algorithm to recover $C$ from $r(C)$, see \cite[\S 3]{CR22}.

\begin{lemma}
    Let $\lambda$ be an infinitesimal parameter of $\GL_n(F)$.    Let $C_1, C_2$ be two $H_{\lambda}$-orbits of $V_{\lambda}$. 
    \begin{enumerate}
        \item [(a)] We have $C_1=C_2$ if and only if $r(C_1)=r(C_2)$.
        \item [(b)] We have $C_1 \geq_C C_2$ if and only if $r(C_1) \geq r(C_2)$.
    \end{enumerate}
\end{lemma}

Given an $L$-parameter $\phi$ of $\GL_n(F)$, the computation of its associated rank matrices $r(\phi)$ from its decomposition is straightforward. We illustrate this process in the example below, noting that the procedure easily generalizes to more complicated cases.

\begin{exmp}\label{exmp rank matrix GL}
    Let $\phi$ be an $L$-parameter of $\GL_n(F)$ of the form
    \[
        \phi = \bigoplus_{k=1}^m \rho\lvert\cdot\rvert^{x_k} \otimes S_{a_k},
    \]
    where each $x_k \in \half{1}\Z$, and $S_{a_k}$ denotes the unique $a_k$-dimensional irreducible representation of $W_F$. Assume that $x_i + \half{a_i} \in \delta + \Z$ for all $i$ and some fixed $\delta \in \{0, \half{1}\}$, so that the infinitesimal parameter $\lambda := \lambda_\phi$ satisfies $\lambda = \lambda_{[\rho']}$ for $\rho' = \rho\lvert\cdot\rvert^{x_1 + \half{a_1 - 1}}$.

    Set
    \[
        e_{[\rho']}^{\max} = \max_{1 \leq k \leq m}\left\{ x_k + \half{a_k - 1} \right\}, \qquad
        e_{[\rho']}^{\min} = \min_{1 \leq k \leq m}\left\{ x_k + \half{1 - a_k} \right\}.
    \]
    For each summand $\rho\lvert\cdot\rvert^{x_k} \otimes S_{a_k}$, associate an $l_{[\rho']}\times l_{[\rho']}$ matrix (where $l_{[\rho']} = e_{[\rho']}^{\max} - e_{[\rho']}^{\min}$) whose entries are given by
    \begin{align}\label{eq r summand}
        r_{[\rho']}(\rho\lvert\cdot\rvert^{x_k} \otimes S_{a_k})_{ij} :=
        \begin{cases}
            1 & \text{if } e_{[\rho']}^{\max} - x_k - \half{a_k - 1} - 1 \leq i \leq j \leq e_{[\rho']}^{\max} - x_k + \half{a_k - 1}, \\[1ex]
            0 & \text{otherwise.}
        \end{cases}
    \end{align}
    The rank matrix for $\phi$ is then obtained by summing over all summands:
    \[
        r_{[\rho']}(\phi) = \sum_{k=1}^m r_{[\rho']}(\rho\lvert\cdot\rvert^{x_k} \otimes S_{a_k}).
    \]
\end{exmp}

 \subsection{\texorpdfstring{Closure ordering and rank matrix for general groups}{}}

Now we consider a general connected reductive group $G$ (or $\mathrm{O}_{2n}$ in which case we may regard ${}^L\mathrm{O}_{2n}={}^L\mathrm{SO}_{2n}$ but recall from \S\ref{sec even orthog} that the equivalence of $L$-parameters differs). We fix an $n$-dimensional representation $\xi$ of ${}^L G$, which is an algebraic homomorphism $\xi: {}^L G \to  \GL_n(\BC)$. It induces a Lie algebra homomorphism $d\xi: \mfr{g}^{\vee} \to \mfr{gl}_n$.
The composition with $\xi$ gives a map $\xi_{\ast}:\Phi_{\lambda}(G) \to \Phi_{\xi \circ \lambda}(\GL_n)$, and we define
\[ r_{\xi}(\phi):= r(\xi \circ \phi).\]
Since $\xi$ and $d\xi$ are continuous, if $\phi_1 \geq_C \phi_2$, then $\xi\circ \phi_1 \geq_C \xi \circ \phi_2$, but the converse is not true unless $\xi$ is an embedding and $\lambda_{\phi_1}=\lambda_{\phi_2}$. This is the case if $G=G_n$ is a classical group and $\xi=\std$. We record this fact in the following lemma.

\begin{lemma}\label{lem closure ordering for classical groups}
    Let $\lambda$ be an infinitesimal parameter of $G_n$. Let $\phi_1, \phi_2 \in \Phi_{\lambda}(G_n)$. Then $\phi_1 \geq_C \phi_2$ if and only if $\std \circ \phi_1 \geq_C \std \circ \phi_2$ as $L$-parameters of $\GL_N(F)$.
\end{lemma}

The following statement determines $r(\xi \circ \widehat{\phi})$, or equivalently the $\GL_n(\BC)$-conjugacy class of $\xi \circ \widehat{\phi}$.

\begin{lemma}\label{lem rank matrix for involution}
Let $\lambda $ be an infinitesimal parameter of $G$. Decompose $\xi \circ \lambda =\bigoplus_{[\rho] \in R} \lambda_{[\rho]}$. Taking any $y \in {}^t C_{\phi}$, for any $[\rho]\in R$, we have
\[ r_{[\rho]}(\xi \circ \widehat{\phi})_{ij}= \max\{  r_{[\rho]}(d\xi(x))_{ij} \ | \ x \in V_{\lambda},\ [ x,y]=0 \}.  \]
\end{lemma}
\begin{proof}
Since $\xi$ is algebraic and
$r_{[\rho]}(d\xi(x))_{ij}$ counts the rank of $d\xi(x)^{j-i+1}$ after restricting to a subspace, the set
\[\{x \in V_{\lambda}\ | \ r_{[\rho]}(d\xi(x))_{ij} \geq t\}\]
is an open set for any $t$.
Let $p:\Lambda_{\lambda} \to V_\lambda$ be the projection to the first coordinate. Then 
\[\{x\in V_{\lambda} \ | \ [x,y]=0 \} \subseteq p( \overline{\Lambda^{\ast}_{{}^tC_{{\phi}}}})=p( \overline{\Lambda_{C_{\widehat{\phi}}}}) \subseteq \overline{p(\Lambda_{C_{\widehat{\phi}}})}= \overline{C_{\widehat{\phi}}},\] 
which implies that
\[  r_{[\rho]}(\xi \circ \widehat{\phi})_{ij} \geq \max\{  r_{[\rho]}(d\xi(x))_{ij} \ | \ x \in V_{\lambda},\ [ x,y]=0 \}.\]

Conversely, note that $p(\Lambda_{{}^tC_{{\phi}}}^{\ast})$ coincides with the union of all $H_{\lambda}$-orbits of $x \in V_{\lambda}$ satisfying $[x, y] = 0$. Indeed, if $z \in p(\Lambda_{{}^tC_{{\phi}}}^{\ast})$, then there exists $y' \in {}^tC_{{\phi}}$ such that $[z, y'] = 0$. For any $h \in H_{\lambda}$ with $h(y') = y$, we see that $0 = [h(z), h(y')] = [h(z), y]$, so $h(z)$ also satisfies $[h(z), y] = 0$. Conversely, if $z = h(z')$ for some $h \in H_{\lambda}$ and $[z', y] = 0$, then $[z, h^{-1}(y)] = 0$, where $h^{-1}(y) \in {}^tC_{{\phi}}$. Thus, $z$ lies in $p(\Lambda_{{}^tC_{{\phi}}}^{\ast})$. 

Now, since
\[
C_{\widehat{\phi}} \subseteq p(\overline{\Lambda_{C_{\widehat{\phi}}}}) = p(\overline{\Lambda^{\ast}_{{}^tC_{{\phi}}}}) \subseteq \overline{p(\Lambda^{\ast}_{{}^tC_{{\phi}}})},
\]
the above reasoning implies that there exists an orbit $C' \subseteq p(\Lambda^{\ast}_{{}^tC_{{\phi}}})$ whose closure contains $\overline{C_{\widehat{\phi}}}$. Therefore,
\[
r_{[\rho]}(\xi \circ \widehat{\phi})_{ij} \leq \max\{ r_{[\rho]}(d\xi(x))_{ij} \mid x \in V_{\lambda},\, [x, y] = 0 \},
\]
which completes the proof of the lemma.
\end{proof}



Recall that two admissible homomorphisms $\phi_1, \phi_2: W_F \times \SL_2(\BC) \to {}^L G$ are considered equivalent if they are conjugate under $G^{\vee}$. By allowing the representation $\xi$ to vary, the preceding lemma identifies the "weak equivalence" class of $\widehat{\phi}$ in the following sense.

\begin{defn}\label{def weak equiv}
    Two admissible homomorphisms $\phi_1, \phi_2: W_F \times \SL_2(\BC) \to {}^L G$ are said to be \emph{weakly equivalent} if, for every finite-dimensional representation $\xi: {}^L G \to \GL_n(\BC)$, the homomorphisms $\xi \circ \phi_1$ and $\xi \circ \phi_2$ are conjugate in $\GL_n(\BC)$. We write $[\phi]_w$ to denote the weak equivalence class of an $L$-parameter $\phi$.
\end{defn}

\begin{remark}\label{rmk weakly equivalent}\ 
    \begin{enumerate}
        \item For $G= \SO_{2n+1}(F)$, $\Sp_{2n}(F)$, and $\RO_{2n}(F)$, let $\std: {}^L G \to \GL_{N}(\BC)$ denote the standard embedding, where $N=2n+1$ if $G=\Sp_{2n}(F)$ and $N=2n$ otherwise. The induced map $\std_*\colon \Phi(G) \to \Phi(\GL_N)$ is injective \textnormal{(see \cite[Theorem 8.1]{GGP12})}. Thus, for these groups, the weak equivalence class of any $L$-parameter $\phi$ consists of a single $L$-parameter. For the remainder of the paper, we will simply write $r(\phi) := r_{\std}(\phi)$ when $G$ is one of these classical groups.
        \item For $G=\RG_2(F)$, consider the sequence of embeddings
        \[
            \RG_2(\BC) \xrightarrow{\iota} \mathrm{Spin}_7(\BC) \xrightarrow{\mathrm{std}} \SO_{7}(\BC).
        \]
        The composition $\mathrm{std}_* \circ \iota_* : \Phi(\RG_2) \to \Phi(\Sp_6)$ is also injective \textnormal{(\cite[Lemma 2.3(iii)]{GS23})}. Consequently, the weak equivalence class of any $L$-parameter $\phi$ of $\RG_2$ also consists of a single $L$-parameter.
        \item For $G=\SO_{2n}$ with $n > 2$, again let $\std: \SO_{2n}(\BC) \to \GL_{2n}(\BC)$ be the standard embedding. In this case, $\std(\phi_1)$ and $\std(\phi_2)$ are conjugate in $\GL_{2n}(\BC)$ if and only if $\phi_2 \in \{\phi_1, \phi_1^{c}\}$, where $\phi_1^c$ denotes the image of $\phi_1$ under the outer automorphism. Therefore, $[\phi] \subseteq \{\phi, \phi^c\}$, and there exist $L$-parameters $\phi$ for which $[\phi] = \{\phi, \phi^c\}$ \textnormal{(see \cite[Proposition 7.3]{Mat24})}. It is worth noting that in this case, $\dim(C_{\phi}) = \dim(C_{\phi^c})$. This observation is crucial for the discussion in \cite[Part IV]{CHLLRX}.
    \end{enumerate}
\end{remark}

Here is another crucial corollary of Lemma \ref{lem rank matrix for involution}. We may compare the closure ordering between $\widehat{(\xi \circ \phi)}$ and 
$\xi \circ (\widehat{\phi})$, which are $L$-parameters of $\GL_n.$ 

\begin{cor}\label{cor hat GL classical}
Suppose that $\phi$ is an $L$-parameter of $G$. Then as $L$-parameters of $\GL_{n}(F)$, we have
\[ \widehat{(\xi \circ \phi)}\geq_C \xi \circ (\widehat{\phi})\]
\end{cor}
\begin{proof}
We compare the rank matrices of these $L$-parameters. Let $y \in C_{\phi}$ be arbitrary. By Lemma \ref{lem rank matrix for involution}, we have
\begin{align*}
r_{[\rho]}(\xi \circ \widehat{\phi})_{ij} &= \max \left\{ r_{[\rho]}(d\xi(z))_{ij} \;\middle|\; z \in \mfr{g}^{\vee},\ \Ad(\lambda(w))z = |w|z,\ [z, y] = 0 \right\}, \\
r_{[\rho]}(\widehat{(\xi \circ \phi)})_{ij} &= \max \left\{ r_{[\rho]}(x)_{ij} \;\middle|\; x \in \mathfrak{gl}_{n},\ \Ad(\xi \circ \lambda(w))x = |w|x,\ [x, d\xi(y)] = 0 \right\}.
\end{align*}
The second formula follows by applying the lemma to the case $H = \GL_n(F)$ with the identity map $\xi_H: H^{\vee} \to H^{\vee}$, and to the $L$-parameter $\xi \circ \phi \in \Phi(H)$. 

On the other hand, since $\xi$ is a homomorphism of Lie groups, we have the inclusion
\[
\left\{ d\xi(z) \;\middle|\; z \in \mfr{g}^{\vee},\ \Ad(\lambda(w))z = |w|z,\ [z,y]=0 \right\} \subseteq \left\{ x \in \mathfrak{gl}_{n} \;\middle|\; \Ad(\xi \circ \lambda(w))x = |w|x,\ [x, d\xi(y)] = 0 \right\}.
\]
Therefore, $r_{[\rho]}(\widehat{(\xi \circ \phi)})_{ij} \geq r_{[\rho]}(\xi \circ \widehat{\phi})_{ij}$, since the former is the maximum over a set containing the latter. This concludes the proof of the corollary.
\end{proof}

See Example \ref{exmp bad parity} below; the inequality in the above corollary may be strict even if $G$ is a classical group and $\xi=\std$.

\section{Reduction of Pyasetskii involution for classical groups}\label{sec Reduction of Pyasetskii involution for classical groups}
Let $G_n$ be quasi-split $\SO_{2n+1}(F), \Sp_{2n}(F)$ or $\RO_{2n}(F)$. 
Let $\std: {}^L G_n \to \GL_N(\BC)$ be the standard embedding. For an $L$-parameter $\phi$ (resp. infinitesimal parameter $\lambda$) of $G_n$, we write $\phi_{\GL}:= \std \circ \phi$ (resp. $\lambda_{\GL}:= \std \circ \lambda$) as an $L$-parameter (resp. infinitesimal parameter) of $\GL_N(\BC)$. The standard embedding $\std$ identifies $G^{\vee}$ as a subgroup of $\GL_N(\BC)$ that preserves a $\epsilon$-symmetric bilinear form $\langle \cdot , \cdot \rangle$ on $W=\BC^{N}$ (taking identity component if $G_n=\Sp_{2n}(F)$). Then 
\[\mfr{g}^{\vee}= \{ g \in \mfr{gl}(W) \ | \ \langle v_1, g v_2 \rangle+ \langle gv_1,  v_2 \rangle=0, \ \forall v_1,v_2 \in W \}.\]
To make the discussion uniform, we define 
\[ \widetilde{H}_{\lambda}:=\begin{cases}
    H_{\lambda} &\text{ if }G_n=\SO_{2n+1}(F) \text{ or }\RO_{2n}(F),\\
    H_{\lambda} \times \{\pm \id_W\} &\text{ if }G_n =\Sp_{2n}(F).
\end{cases} \]
Observe that the natural map $V_{\lambda}/H_{\lambda} \to V_{\lambda}/ \widetilde{H}_{\lambda}$ is a bijection in all cases and does not affect the computation of the Pyasetskii involution. With this notation, we have the following descriptions:
\begin{align*}
    V_{\lambda} &= \left\{ g \in V_{\lambda_{\GL}} \;\middle|\; \langle v_1, g v_2 \rangle + \langle gv_1, v_2 \rangle = 0 \text{ for all } v_1, v_2 \in W \right\}, \\
    \widetilde{H}_{\lambda} &= \left\{ h \in H_{\lambda_{\GL}} \;\middle|\; \langle h v_1, h v_2 \rangle = \langle v_1, v_2 \rangle \text{ for all } v_1, v_2 \in W \right\}.
\end{align*}

\subsection{Decomposition of Vogan variety}\label{sec decomp}
Fix an infinitesimal parameter $\lambda$ of $G_n$. Recall from \S\ref{sec closure ordering} that we have the decompositions
\[
    \lambda_{\GL} = \bigoplus_{i \in I} \rho_i = \bigoplus_{[\rho] \in R} \bigoplus_{i \in I_{[\rho]}} \rho_i,
\]
where each $W_i \subseteq W$ denotes the underlying space of $\rho_i$, and $W_{[\rho]} := \bigoplus_{i \in I_{[\rho]}} W_i$.
Since $\lambda$ preserves the non-degenerate bilinear form $\langle \cdot, \cdot \rangle$, we can further decompose the index set $R$ as $R = R_{nsd} \sqcup R_{sd}$, where
\[
    R_{nsd} = \{ [\rho] \in R \mid [\rho^\vee] \neq [\rho] \}, \qquad R_{sd} = \{ [\rho] \in R \mid [\rho^\vee] = [\rho] \}.
\]
For $[\rho] \in R_{nsd}$, the restriction of $\langle \cdot, \cdot \rangle$ to $W_{[\rho]} \oplus W_{[\rho^\vee]}$ is still non-degenerate; for $[\rho] \in R_{sd}$, the restriction to $W_{[\rho]}$ remains non-degenerate.

For each $[\rho]$, define the associated infinitesimal parameter of a general linear group by
\begin{align}\label{eq rho part lambda}
    \lambda_{\GL, [\rho]} :=
    \begin{cases}
        \displaystyle\bigoplus_{i \in I_{[\rho]} \sqcup I_{[\rho^\vee]}} \rho_i &\text{if } [\rho] \in R_{nsd},\\[2ex]
        \displaystyle\bigoplus_{i \in I_{[\rho]}} \rho_i &\text{if } [\rho] \in R_{sd}.
    \end{cases}
\end{align}
Let $d_{\rho}$ denote the dimension of $\lambda_{\GL, [\rho]}$.

In what follows, we construct for each $[\rho]$ an infinitesimal parameter $\lambda_{[\rho]}$ associated to a suitable quasi-split classical group $G_{[\rho]}$. Note that $G_{[\rho]}$ may differ in type from $G_n$. By \cite[Theorem 8.1]{GGP12}, to specify $\lambda_{[\rho]}$, it suffices to specify the corresponding $L$-parameter $ \std \circ \lambda_{[\rho]}$ for general linear groups.

Suppose $G_n= \SO_{2n+1}$. Then $d_{\rho}$ must be even since it preserves a symplectic form. Then $G_{[\rho]}\cong \SO_{d_{\rho}+1}(F)$ and $\lambda_{[\rho]}$ is determined by $\std \circ \lambda_{[\rho]}= \lambda_{\GL, [\rho]}$. 
Suppose $G_n = \Sp_{2n}(F)$ or $\RO_{2n}(F)$. 
If $d_{\rho}$ is odd, then $G_{[\rho]} = \Sp_{d_{\rho}-1}(F)$, and $\lambda_{[\rho]}$ is characterized by $\std \circ \lambda_{[\rho]} = \det(\lambda_{\GL, [\rho]}) \otimes \lambda_{\GL, [\rho]}$. If $d_{\rho}$ is even, then $G_{[\rho]}$ is the quasi-split even orthogonal group $\RO(V)$, where the discriminant of $V$ coincides with $\det(\lambda_{\GL, [\rho]})$. In this case, $\lambda_{[\rho]}$ is determined by $\std \circ \lambda_{[\rho]} = \lambda_{\GL, [\rho]}$.

Take a subset $R_{nsd/2} \subset R_{nsd}$ such that $R_{nsd}= \cup_{[\rho] \in R_{nsd/2}} \{[\rho], [\rho^{\vee}]\}$. Then by the description of $V_{\lambda}, \widetilde{H}_{\lambda}$ and \eqref{eq decomp rho GL}, we obtain a decomposition
\begin{align}\label{eq decomp rho classical}
     V_{\lambda}= \bigtimes_{[\rho] \in R_{nsd/2}\sqcup R_{sd}} V_{\lambda_{[\rho]}} ,\ \ \widetilde{H}_{\lambda} =\bigtimes_{[\rho] \in R_{nsd/2}\sqcup R_{sd}} \widetilde{H}_{\lambda_{[\rho]}}.
\end{align}
We define $\phi_{[\rho]} \in \Phi_{\lambda_{[\rho]}}(G_{[\rho]})$ accordingly.

\begin{defn}\label{def rho part}
    For any $L$-parameter $\phi \in \Phi_{\lambda}(G_n)$ and $[\rho]\in R$, we define its $\rho$-isotypic part $\phi_{[\rho]} \in \Phi_{\lambda_{[\rho]}}(G_{[\rho]})$ by 
    \[ C_{\phi}= \bigtimes_{[\rho] \in R_{nsd/2}\sqcup R_{sd} } C_{\phi_{[\rho]}}\]
    under the decomposition \eqref{eq decomp rho classical}. We say $\phi$ is isotypic if $\phi=\phi_{[\rho]}$ for some $[\rho]$.
\end{defn}
 The following lemma reduces the computation of Pyasetskii involution for $\phi$ to $\phi_{[\rho]}$.

\begin{lemma}\label{lem phihat}
Let $\phi \in \Phi_{\lambda}(G_n)$ with Pyasetskii involution $\widehat{\phi}$. For any $[\rho]\in R$, define $\phi_{[\rho]}, (\widehat{\phi})_{[\rho]} \in \Phi_{\lambda_{[\rho]}}(G_{[\rho]})$ as in Definition \ref{def rho part}. Then $(\widehat{\phi})_{[\rho]} $ is the Pyasetskii involution of $\phi_{[\rho]}$.
\end{lemma}
\begin{proof}
    This is a direct consequence of the decomposition \eqref{eq decomp rho classical}.
\end{proof}

\subsection{The non-selfdual case}
In this subsection, we compute the Pyasetskii involution for $\phi_{[\rho]}$ when $[\rho]\neq [\rho^{\vee}]$. Note that in this case $G_{[\rho]}$ is either $\SO_
{2n+1}(F)$ or $\RO_{2n}(F)$. Moreover, $G_{[\rho]}$ must be split since $\phi_{[\rho]}$ factors through a Siegel parabolic subgroup of $G^{\vee}_{[\rho]}$.

\begin{prop}\label{prop nsd}
    Suppose that $\phi \in \Phi(G_n)$ satisfies that $\phi=\phi_{[\rho]}$ and $G_{n}=G_{[\rho]}$ for some $\rho$ such that $[\rho] \neq [\rho^{\vee}]$. Write $\std\circ \phi= \phi_{1} \oplus \phi_{1}^{\vee}$, where $\phi_{1}$ is an $L$-parameter of $\GL_n(F)$ of the form
    \[ \phi_{1}= \bigoplus_{i\in I} \rho\lvert\cdot \rvert^{x_i} \otimes S_{a_i} \]
    with $x_i \in \Z$. Then $\widehat{\phi}$ is determined by $\std \circ \widehat{\phi}= \widehat{\phi_{1}}\oplus \widehat{\phi_1}^{\vee}$. Here $\widehat{\phi_1}$ can be computed by M{\oe}glin--Waldspurger's algorithm $($\cite[\S 2]{MW86}$)$.
\end{prop}
\begin{proof}
The standard embedding $\std: {}^L G_n \to \GL(W) $ realizes $G_n^{\vee}$ as the isometry group of a non-degenerate $\epsilon$-symmetric bilinear form $\langle \cdot,\cdot\rangle$ on $W \cong \BC^{2n}$. Decompose $W= W_{\rho} \oplus W_{\rho^\vee}$, where $W_{\rho}$ (resp. $W_{\rho^{\vee}}$) is the underlying space of the representation $\phi_1$ (resp. $\phi_1^{\vee}$) of $W_F \times \SL_2(\BC)$. Note that since $[\rho]\neq [\rho^{\vee}]$, $W_{\rho}$ and $W_{[\rho]}$ are maximal isotropic and $W= W_{\rho} \oplus W_{\rho^\vee}$ is a complete polarization. We let $P=LU$ be the Siegel parabolic subgroup of $G_n^{\vee}$ that preserves the maximal isotropic space $W_{\rho^{\vee}}$, and $L$ is the Levi subgroup that also preserves $W_{\rho}$. Note that $L \cong \GL(W_{\rho})$.

Define an infinitesimal parameter $\lambda_1$ of $\GL_n(F)$ such that $\std \circ \lambda_{\phi}= \lambda_1 \oplus \lambda_1^{\vee}$ and $\phi_{1} \in \Phi_{\lambda_1}(\GL_n(F))$. Then since the infinitesimal parameter $\lambda_{\phi}$ has image in $L \subset G_{n}^{\vee} \subset \GL(W)$,
we obtain injections
\begin{align*}
    V_{\lambda_1} \hookrightarrow V_{\lambda_{\phi}} \hookrightarrow V_{\std \circ \lambda_{\phi}},\ \ 
    H_{\lambda_1} \hookrightarrow \widetilde{H}_{\lambda_{\phi}} \hookrightarrow H_{\std \circ \lambda_{\phi}}.
\end{align*}
On the other hand, by \eqref{eq decomp rho GL}, we see that $V_{\std \circ \lambda_{\phi}} \subseteq \End(W_{\rho}) \times \End(W_{\rho^{\vee}})$ and $H_{\std \circ \lambda_{\phi}} \subseteq \Aut(W_{\rho}) \times \Aut(W_{\rho^{\vee}})$. As a consequence, $V_{{\lambda}_{\phi}} \subseteq \textrm{Lie}(L)\cong \mfr{gl}(W_{\rho})$ and $H_{\lambda_{\phi}} \subseteq L$. Therefore, the injections $V_{\lambda_1} \hookrightarrow V_{\lambda_{\phi}}$ and $H_{\lambda_1} \hookrightarrow \widetilde{H}_{\lambda_{\phi}} $ are indeed bijections. We conclude that the Pyasetskii involution on $H_{\lambda_{\phi}}/V_{\lambda_{\phi}}$ can be computed from that of $H_{\lambda_1}/V_{\lambda_1}$, which completes the proof of the proposition.
\end{proof}

\subsection{Good and bad parity}

Now let us consider the case $\phi_{[\rho]}$ with $[\rho]=[\rho^{\vee}]$. 
Choose any representative $\rho \in [\rho]$. We can write $\rho= \rho^{u} \lvert\cdot\rvert^{x}$, where $\rho^u$ is unitary and $x \in \mathbb{R}$. Since $[\rho]=[\rho^{\vee}]$, it follows that $\rho^u$ is self-dual and $x \in \frac{1}{2}\mathbb{Z}$. The isomorphism $\rho^{u} \cong (\rho^u)^{\vee}$ induces a non-degenerate $\epsilon_{\rho^{u}}$-symmetric bilinear form on the underlying space of $\rho^{u}$. We define
\begin{align}\label{eq epsilon rho}
    \epsilon_{[\rho]} := \epsilon_{\rho^{u}}(-1)^{2x}.
\end{align}
It is straightforward to check that this definition does not depend on the choice of representative $\rho$.

\begin{defn}\label{def gp bp}
    Suppose that $\phi = \phi_{[\rho]}$ with $[\rho] = [\rho^{\vee}]$. Via the standard embedding $\std: {}^L G_{n} \to \GL(W)$, we may view ${}^L G_{n}$ as a subgroup of the isometry group for a non-degenerate $\epsilon$-symmetric bilinear form $\langle \cdot, \cdot \rangle$ on $W$.
    We say that $\phi$ is of \emph{good parity} if $\epsilon \epsilon_{[\rho]} = 1$, and of \emph{bad parity} if $\epsilon \epsilon_{[\rho]} = -1$.

    Similarly, we say that an infinitesimal parameter $\lambda$ of $G_n$ is of good (resp., bad) parity if it is of good (resp., bad) parity as an $L$-parameter. It is evident that $\phi$ is of good or bad parity if and only if its infinitesimal parameter $\lambda_{\phi}$ is of good or bad parity.
\end{defn}
Though not needed in this paper, we can define a general $L$-parameter $\phi$ of $G_n$ is of good parity if $\phi_{[\rho]}$ is of good parity for any $[\rho] \in R$. This is equivalent to the representations in $\Pi_{\phi}$ are all of good parity as in \cite[Definition 3.1]{HJLLZ24} and \cite[Definition 4.5.5]{LM25}.

\begin{remark}\label{rmk gp} 
We explain the main differences between the good and bad parity cases.
\begin{enumerate}
    \item If $\phi$ is an $L$-parameter of $G_n$, then $\std\circ \phi$ preserves a non-degenerate bilinear form. Hence, $\std\circ \phi$ must be selfdual. 
    The converse is true in the good parity case: Let $\lambda$ be a good parity infinitesimal parameter of $G_n$. For any $\phi_1\in \Phi_{\lambda_{\GL}}(\GL_N(F))$, if $\phi_1$ is selfdual, then there exists a $\phi \in \Phi_{\lambda}(G_n)$ such that $\phi_1= \std \circ \phi$.
    However, this is not true in the bad parity case as shown in Example \ref{exmp bad parity} below.
    \item On the other hand, if $\phi$ is of bad parity, then the centralizer of the image of $\phi$ in $G_n^{\vee}/Z(G_n^{\vee})$ is connected \textnormal{(see \cite[\S 4]{GGP12})}. As a consequence, the $L$-packet of $\Pi_{\phi}$ is always a singleton. 
\end{enumerate}    
\end{remark}

\begin{exmp}\label{exmp bad parity}
    Let $\rho$ denote the trivial representation, and consider an infinitesimal parameter $\lambda$ of $\SO_{7}(F)$ such that
\[
    \lambda_{\GL}= \rho\, \lvert\cdot\rvert^{1} \oplus (\rho)^{4} \oplus \rho\,\lvert\cdot\rvert^{-1},
\]
where $(\rho)^4$ indicates four copies of $\rho$. The set $\Phi_{\lambda_{\GL}}(\GL_6(F))$ consists of five $L$-parameters:
\begin{align*}
    \phi_0 &= \rho \otimes S_3 \oplus (\rho \otimes S_1)^{3}, \\
    \phi_1 &= \rho\,\lvert\cdot\rvert^{1/2} \otimes S_2 \oplus \rho\,\lvert\cdot\rvert^{-1/2} \otimes S_2 \oplus (\rho \otimes S_1)^{2}, \\
    \phi_2 &= \rho\,\lvert\cdot\rvert^{1/2} \otimes S_2 \oplus \rho\,\lvert\cdot\rvert^{-1} \otimes S_1 \oplus (\rho \otimes S_1)^{3}, \\
    \phi_3 &= \rho\,\lvert\cdot\rvert^{-1/2} \otimes S_2 \oplus \rho\,\lvert\cdot\rvert^{1} \otimes S_1 \oplus (\rho \otimes S_1)^{3}, \\
    \phi_4 &= \lambda_{\GL} \otimes S_1.
\end{align*}
Among these, $\phi_0$, $\phi_1$, and $\phi_4$ are self-dual, while $\phi_2^{\vee} = \phi_3$. However, $\phi_0$ does not preserve a skew-symmetric form. Therefore, $\Phi_{\lambda}(\SO_{7}(F)) = \{\phi^1,\,\phi^4\}$, where $\std \circ \phi^1 = \phi_1$ and $\std \circ \phi^4 = \phi_4$. These correspond to the unique open and closed orbits, respectively, so we have $\widehat{\phi^1} = \phi^4$. On the other hand, for $\Phi(\GL_6(F))$, we have
\[
    \widehat{\phi_0} = \phi_4, \quad \widehat{\phi_1} = \phi_1, \quad \widehat{\phi_2} = \phi_3.
\]
\end{exmp}

\subsection{Unramification}
To conclude this section, we reduce the computation of $\widehat{\phi}$ in both the good and bad parity cases to that of an unramified $L$-parameter $\widehat{\phi^{\ur}}$ (i.e., trivial on the inertia subgroup $I_F$) for a suitably chosen group $G^{\ur}_n$.

\begin{defn}\label{defn unramification}
    Suppose that $\phi=\phi_{[\rho]}$ with $[\rho]= [\rho^{\vee}]$. Via the standard embedding $\std: {}^L G_{n} \to \GL(W)$, we identify ${}^L G_{n}$ as a subgroup of an isometry group of a $\epsilon$-symmetric non-degenerate bilinear form $\langle \cdot,\cdot\rangle$ on $W$. Write
    \[ \std\circ \phi= \bigoplus_{i \in I} \rho^u \lvert\cdot\rvert^{x_i} \otimes S_{a_i},\]
    where $\rho^u$ is selfdual and $x_i \in \half{1}\Z$.
    \begin{enumerate}
        \item Define $G^{\ur}_n$ to be  split $\SO_{2m+1}$, $\Sp_{2m}$ or $\RO_{2m}$  whose complex dual group is a subgroup of the isometry group of an $\epsilon\epsilon_{\rho^u}$-symmetric non-degenerate bilinear form on $W^{\ur}\cong \BC^{\sum_{i\in I} a_i}$.
        \item Define $\phi^{\ur}$ to be the $L$-parameter of $G^{\ur}_n$ such that
        \[ \std \circ \phi^{\ur}= \bigoplus_{i \in I}\lvert\cdot\rvert^{x_i} \otimes S_{a_i}.\]
    \end{enumerate}
\end{defn}
By definition, $\phi$ has good parity if and only if $\phi^{\ur}$ has good parity.

\begin{prop}\label{prop unramification}
     Suppose that $\phi=\phi_{[\rho]}$ with $[\rho]= [\rho^{\vee}]$. We have $(\widehat{\phi})^{\ur}= \widehat{(\phi^{\ur})}$.
\end{prop}
\begin{proof}
    With the notation from Definition \ref{defn unramification}, we have $\phi_{\GL} = \rho \otimes \phi_{\GL}^{\ur}$ (tensor product of representations of $W_F$), so $W = W_{\rho} \otimes W^{\ur}$. By Schur's lemma, the map $f \mapsto \id_{W_{\rho}} \otimes f$ from $\End(W^{\ur})$ to $\End(W)$ induces natural isomorphisms
    \[
        V_{\lambda^{\ur}} \xrightarrow{\;\sim\;} V_{\lambda}, \qquad \widetilde{H}_{\lambda^{\ur}} \xrightarrow{\;\sim\;} \widetilde{H}_{\lambda}.
    \]
    The proposition follows.
    \end{proof}

\section{A lemma on involutions on finite partially ordered sets}\label{sec A lemma on involutions on finite partially ordered sets}
We state and prove a simple lemma on involutions on finite partially ordered sets. It will be used to identify the Pyasetskii involution with other known involutions by comparing their images with respect to the closure ordering. In the bad parity case, it saves us from modifying the most technical parts of the proof in \cite[Proposition II.6]{MW86} from $\GL_n$ to classical groups.

\begin{lemma}\label{lem involution poset}
Let $(S, \geq)$ be a finite partially ordered set. Suppose that there are two involutions $\iota_1, \iota_2$ on $S$ such that for any $s \in S$, $ \iota_1(s) \geq \iota_2(s).$ Then $\iota_1=\iota_2$.
\end{lemma}
\begin{proof}
Suppose for a contradiction that the set
\[ S':=\{ s \in S \mid \iota_1(s) \neq \iota_2(s) \} \]
is non-empty, and choose a minimal element $s \in S'$. Then
\[ s = \iota_1 \circ \iota_1(s) \geq \iota_2 \circ \iota_1(s). \]
Equality cannot hold, since applying $\iota_2$ to both sides would give $\iota_2(s) = \iota_1(s)$. Thus $\iota_2 \circ \iota_1(s) \notin S'$ by the minimality of $s$, and so
\[ \iota_1(s) = \iota_2 \circ \iota_2 \circ \iota_1(s) = \iota_2(\iota_2 \circ \iota_1(s)) = \iota_1(\iota_2 \circ \iota_1(s)). \]
Applying $\iota_2 \circ \iota_1$ to both sides yields $\iota_2(s) = \iota_1(s)$, a contradiction. 
\end{proof}
\section{Aubert-Zelevinsky involution}\label{sec AZ}

We recall the definition of Aubert-Zelevinsky involution. Let $\RG$ be a connected reductive group defined over $F$. For any element $\Pi$ in the Grothendieck group $K\Pi(G)$, we define
\begin{align}\label{eq def AZ}
    \AZ(\Pi):=\sum_P (-1)^{\mathrm{dim}(A_P)}\mathrm{Ind}_{P}^{G}(\textrm{Jac}_P(\Pi))
\end{align}
Here the sum is over all standard parabolic subgroups $P$ of $G$ and $A_P$ is the maximal split torus of the center of the Levi subgroup of $P.$
It is proved in \cite{Aub95} that if $\pi\in \Pi(G)$, then there exists a $\varepsilon\in\{\pm 1\}$ such that $\AZ(\pi)=\varepsilon \widehat{\pi}$ is an irreducible representation of coefficient $1$. Moreover, $\widehat{\widehat{\pi}}=\pi$.
We call $\widehat{\pi}$ the Aubert-Zelevinsky dual or Aubert-Zelevinsky involution of $\pi.$

\section{The good parity case}\label{sec The good parity case}
We compute $\widehat{\phi}$ when $\phi$ is isotypic and of good parity in this section.
\begin{prop}\label{prop involution of orbit gp}
    Suppose $\lambda$ is an isotypic infinitesimal parameter of $G_n$ of good parity. For any $\phi \in \Phi_{\lambda}(G_n)$, the following holds.
    \begin{enumerate}
        \item [(a)] There exists a unique $L$-parameter $\widetilde{\phi} \in \Phi_{\lambda}(G_n)$ such that 
        \[\widehat{(\std\circ \phi)}= \std \circ \widetilde{\phi}.\]
        \item [(b)]We have $\widehat{\phi}= \widetilde{\phi}$.
    \end{enumerate}
\end{prop}
\begin{proof}
    We first prove Part (b) assuming Part (a). Recall that the map $\phi \mapsto \std \circ \phi$ gives an order-preserving injection from $\Phi_{\lambda}(G_n)$ into $\Phi_{\lambda_{\GL}}(\GL_N(F))$. Since 
    \[\std \circ \widetilde{\widetilde{\phi}}= \widehat{(\std \circ \widetilde{\phi})}= \widehat{\widehat{(\std \circ \phi)}} = \std \circ \phi, \]
    we see that $\widetilde{\widetilde{\phi}}= \phi$. Namely, the map $ \phi \mapsto \widetilde{\phi}$ is an involution on $\Phi_{\lambda}(G_n)$. Moreover, Corollary \ref{cor hat GL classical} implies that
    \[\std \circ \widetilde{\phi} =\widehat{(\std\circ \phi)} \geq_C \std \circ \widehat{\phi}.\]
    for any $\phi\in \Phi_{\lambda}(G_n)$. Thus, Lemma \ref{lem closure ordering for classical groups} implies that $\widetilde{\phi}\geq_C \widehat{\phi}$ for any $\phi\in \Phi_{\lambda}(G_n)$.
        Therefore, Lemma \ref{lem involution poset} implies that $\widehat{\phi}=\widetilde{\phi}$ for any $\phi$, which proves Part (b).

    For Part (a), since $\lambda$ is of good parity, it suffices to show that $\phi_{\GL}$ is selfdual if and only if $\widehat{\phi_{\GL}}$ is selfdual (see Remark \ref{rmk gp}(1)). This follows from the following facts:
    \begin{itemize}
        \item  As shown in \cite{MW86}, the Pyasetskii involution for $L$-parameters and the Aubert-Zelevinsky involution for representations of $\GL_n(F)$ commute via the local Langlands correspondence.
        \item The Aubert-Zelevinsky involution commutes with taking contragredients, since both parabolic induction and the formation of Jacquet modules in \eqref{eq def AZ} commute with contragredients.
    \end{itemize}
    This completes the proof of the proposition. 
\end{proof}

\begin{remark}
    The above proposition fails without assuming $\lambda$ is of good parity. Indeed, in Example \ref{exmp bad parity}, we have $\widehat{\std \circ \phi^4}=\widehat{\phi_4}= \phi_0$, which does not come from an $L$-parameter of $\SO_{7}(F)$.
\end{remark}

\section{The bad parity case}\label{sec bad}
In this section, we show that Pyasetskii involution in the bad parity case is given by the formula for the Aubert-Zelevinsky involution for bad parity representations in \cite[\S 5.2]{LM25} (see Algorithm \ref{algo LM25} below). More precisely, we prove the following theorem.

\begin{thm}\label{thm bad}
Suppose $\phi$ is an isotypic $L$-parameter of $G_n$ of bad parity. Then \[\widehat{\phi} = \AZ(\phi),\] where $\AZ(\phi)$ is defined in Definition~\ref{def AZ phi bad} below and can be computed explicitly by Algorithm~\ref{algo LM25}.
\end{thm}

Our proof of the theorem proceeds by suitably adapting half of the argument of \cite{MW86} from the general linear group setting to the context of classical groups. Thanks to Lemma \ref{lem involution poset}, we can avoid the other half of their argument.

\subsection{\texorpdfstring{Algorithm for $\AZ(\phi)$}{}}
In this subsection, we review Lanard-M{\'i}nguez's formula for computing $\widehat{\pi}$ in the case of purely bad parity representations of $\Sp_{2n}(F)$ and $\SO_{2n+1}(F)$. We also introduce the definition of $\AZ(\phi)$. To begin, we recall some basic notation regarding segments.

\begin{defn}\label{def segments}\ 
    \begin{enumerate}
        \item A \emph{segment} is a set of supercuspidal representations of a general linear group
        \[\Delta=\{\rho\lvert\cdot\rvert^{b},\rho\lvert\cdot\rvert^{b+1},\ldots, \rho\lvert\cdot\rvert^{e}\},\]
        where $b,e\in\mathbb{R},$ $e-b \in \Z_{\geq 0}$, and $\rho$ is an irreducible unitary supercuspidal representation of some $\GL_d(F)$. We shall write $\Delta=[b,e]_{\rho}$ and write $b=b(\Delta), e=e(\Delta)$ and $l(\Delta):=e-b+1$.
        \item For a segment $\Delta=[b,e]_{\rho}$, we define its contragredient as $\Delta^{\vee}:= [-e,-b]_{\rho^{\vee}}$.
        \item For a segment $\Delta=[b,e]_{\rho}$, define $\Delta^-:=[b,e-1]_{\rho} $ and ${}^-\Delta:=[b+1,e]_{\rho} $
        \item We say $\Delta'=[b',e']_{\rho'}$ \emph{precedes} $\Delta=[b,e]_{\rho}$ if $\rho'=\rho, \ b'-b \in \Z$ and
        \[ b'<b,\ e'< e,\ b\leq e'+1.\]
        We write $\Delta' \prec \Delta$ in this case.
        \item A multi-segment is a multi-set of segments $\mfr{m}=\{[b_i,e_i]_{\rho_i}\}_{i \in I}$. For each irreducible unitary supercuspidal representation $\rho$, we consider the ``$\rho$-part" of $\mfr{m}$ by setting $I_{\rho}=\{i \in I \ | \ \rho_i \cong \rho\}$ and defining
        \[ \mfr{m}_{\rho}:=\{[b_i,e_i]_{\rho_i}\}_{i \in I_{\rho}}. \]
         \item  We shall identify $L$-parameters of general linear groups and multi-segments via the following bijection.
        \[ \phi=\bigoplus_{i\in I} \rho_i\lvert\cdot \rvert^{x_i} \otimes S_{a_i} \longmapsto \mfr{m}_{\phi}= \{ [(1-a_i)/2+x_i,(a_i-1)/2+x_i ]_{\rho_i} \}_{i \in I}.\]
    \end{enumerate}
\end{defn}
If $\mfr{m}_1, \mfr{m}_2 $ are two multi-segments, we let $\mfr{m}_1 + \mfr{m}_2$ or $\mfr{m}_1 \sqcup \mfr{m}_2$ be the disjoint union of multi-sets of segments. We use $\Sigma$ for disjoint union of multi-sets over certain index set. For instance, we have
$\mfr{m}= \sum_{\rho} \mfr{m}_{\rho}= \sqcup_{\rho} \mfr{m}_{\rho}.$
If $\mfr{m}_2$ is a multi-subset of $\mfr{m}_1$, we let $ \mfr{m}_1- \mfr{m}_2$ denote the multi-segment characterized by $(\mfr{m}_1- \mfr{m}_2)+ \mfr{m}_2= \mfr{m}_1$.

We consider multi-segments that come from isotypic $L$-parameters of bad parity.

\begin{defn}\label{def Mseg bad}
   For selfdual $\rho$, we let $\MSeg_{\rho,bad}$ be the set of multi-segments $\mfr{m}=\{\Delta_i=[b_i,e_i]_{\rho}\}_{i \in I}$ satisfying the following conditions. 
   \begin{enumerate}
       \item [(i)] There exists a $\delta \in \{0, \half{1}\}$ such that $b_i \in \delta+\Z$ for all $i \in I$.
       \item [(ii)] There is an involution $i \mapsto i^{\vee}$ on $I$ such that for any $i \in I,$
    \begin{itemize}
        \item $\Delta_{i^{\vee}}= (\Delta_{i})^{\vee}$, and
        \item $i \neq i^{\vee}$.
    \end{itemize}
   \end{enumerate}
    We shall write $\mfr{m} \in \MSeg_{\rho,bad}$ or $(\mfr{m}, \vee) \in \MSeg_{\rho,bad}$.   Finally, we let 
    \[\MSeg_{bad}:= \{ \mfr{m}\ | \ \mfr{m}_{\rho}\in \MSeg_{\rho,bad} \text{ for all selfdual }\rho \}.\]
\end{defn}

Now we state the algorithm of Lanard-M{\'i}nguez on multi-segments.

\begin{algo}[{\cite[\S 5.2]{LM25}}]\label{algo LM25}
    Let $\mfr{m}=\{\Delta_i\}_{i \in I} \in \MSeg_{\rho,bad}$. We define $\AZ_{bad}(\mfr{m})$ inductively as follows.

    \begin{enumerate}
        \item [Step 1:] If $\mfr{m}= \emptyset$, then $\AZ_{bad}(\mfr{m}):= \emptyset$.
        \item [Step 2:] Let $d:= \max_{i \in I}\{e(\Delta_i)\}$ and choose an $ i_0 \in I$ such that
        \[ e(\Delta_{i_0}) = d,\ b(\Delta_{i_0})= \max\{ b(\Delta_i)\ | \ e(\Delta_i)=d \}.\]
        \item [Step 3:] For $k \geq 1$, choose an index $i_k$ as follows. Let $I_{k}$ be the subset of index $i \in I$ satisfying the following conditions.
        \begin{enumerate}
            \item [(a)] $\Delta_{i} \prec \Delta_{i_{k-1}}$ and $e(\Delta_{i})= d-k$.
            \item [(b)] If $\Delta_{i}^{\vee}= \Delta_{i_j}$ for some $0 \leq j \leq k-1$, then there exists an $i' \in I \setminus \{i\}$ such that $\Delta_{i'}=\Delta_i$.
        \end{enumerate}
        If the set $I_{k}$ is non-empty, then take any $i_k \in I_{k}$ such that
        \[ b(\Delta_{i_k})= \max_{i \in I_{k}} \{b(\Delta_i)\}.\]
        If $I_{k}$ is empty, proceed to Step 4.
        \item [Step 4:] Let $\{i_0,\ldots,i_r\}$ be the sequence of indices chosen in Step 2,3.
        Define the multi-segment $\mfr{m}^{\#}$ by
        \[ \mfr{m}^{\#}= \mfr{m}- \sum_{j=0}^r  \{\Delta_{i_j}, \Delta_{i_j}^{\vee}\} 
        + \sum_{j=0}^r \{ \Delta_{i_j}^- ,{}^-(\Delta_{i_j}^{\vee}) \} . \]
        Then define 
        \[ \AZ_{bad}(\mfr{m}):= \{[d-r,d]_{\rho},[-d,-d+r]_{\rho}\}+ \AZ_{bad}(\mfr{m}^{\#}). \]
    \end{enumerate} 

    Finally, for general $\mfr{m}\in \MSeg_{bad}$, we define $\AZ_{bad}(\mfr{m}):= \sum_{\rho} \AZ_{bad}(\mfr{m}_{\rho})$.
\end{algo}

We record the translation of the algorithm from multi-segments to $L$-parameters.

\begin{defn}\label{def AZ phi bad}\ 
\begin{enumerate}
    \item Suppose $\phi$ is an $L$-parameter of $G_n=\SO_{2n+1}(F)$ or $\RO_{2n}(F)$ such that $\phi_{[\rho]}$ (see Definition \ref{def rho part}) is of bad parity for any $[\rho]$. Then we define $\AZ(\phi)$ to be the unique $L$-parameter of $G_n$ such that
    \[ \mfr{m}_{\std \circ \AZ(\phi)}= \AZ_{bad}(\mfr{m}_{\std \circ \phi} ).\]
    \item Suppose $\phi$ is an $L$-parameter of $\Sp_{2n}(F)$ such that $\phi_{[\rho]}$ is of bad parity for $[\rho]\neq [1]$ and $\phi_{[1]}=1 \otimes S_1$ is the trivial $L$-parameter of $\Sp_{0}(F)$. Write $\mfr{m}_{\std \circ \phi}= \mfr{m}' \sqcup \{[1,1]_{1}\}$ where $\mfr{m}' \in \MSeg_{bad}$. Then we define $\AZ_{bad}(\phi)$ to be the unique $L$-parameter of $G_n$ such that
    \[ \mfr{m}_{\std \circ \AZ(\phi)}= \AZ_{bad}(\mfr{m}')\sqcup \{[1,1]_{1}\}.\]
\end{enumerate}
\end{defn}

\begin{thm}[{\cite[Theorem 5.4.1]{LM25}}]\label{thm LM25}
    Suppose $\phi$ is an $L$-parameter of $G=\Sp_{2n}(F)$ or $\SO_{2n+1}(F)$ satisfying the condition in Definition \ref{def AZ phi bad}. Let $\pi$ be the unique member of the $L$-packet $\Pi_{\phi}(G)$. Then
    \[ \Pi_{\AZ(\phi)}(G)= \{ \widehat{\pi}\},\]
    where $\widehat{\pi}$ is the Aubert-Zelevinsky involution of $\pi$.
\end{thm}

\begin{cor}
    For any isotypic $L$-parameter of $G_n=\Sp_{2n}(F)$ or $\SO_{2n+1}(F)$ or $\RO_{2n}(F)$ of bad parity, we have $\AZ(\AZ(\phi))=\phi$. 
\end{cor}
\begin{proof}
    If $G_n=\Sp_{2n}(F)$ or $\SO_{2n+1}(F)$, this is a direct consequence of Theorem \ref{thm LM25}. If $G_{n}=\RO_{2n}(F)$, then we let $\phi'$ be the $L$-parameter of $\Sp_{2n}(F)$ such that
    \[ \std \circ \phi' = \std \circ \phi \oplus 1 \otimes S_1.
    \]
   Since $\AZ(\AZ(\phi'))=\phi'$ by Theorem \ref{thm LM25}, we see that $\AZ(\AZ(\phi))=\phi$ by Definition \ref{def AZ phi bad}.
\end{proof}

\begin{remark}
    It is a natural expectation that Theorem \ref{thm LM25} also holds for split $\RO_{2n}(F)$ as well. We shall see in \S \ref{sec ABV} that this is a consequence of an expected property of ABV-packets (Conjecture \ref{conj AZ}).
\end{remark}

We state a special case of Theorem \ref{thm bad}.
\begin{exmp}\label{ex bad open}
    Let $d$ be a half integer and $r \in \Z_{\geq 0}$. Consider the infinitesimal parameter
    \[\lambda:= \bigoplus_{i=0}^r (\lvert\cdot\rvert^{d-r} \oplus \lvert \cdot \rvert^{r-d})\]
    of $\RO_{2n}$ (if $d\in \half{1}+\Z$) or $\SO_{2n+1}$ (if $d \in \Z$) of bad parity. Consider
    \[ \mfr{m}_0= \mfr{m}_{\lambda_{\GL}}= \bigcup_{i=0}^r\{[d-i, d-i],[i-d,i-d]\},\]
    which corresponds to 
    the unique closed orbit in the Vogan variety $V_{\lambda}$. Then the algorithm gives
    \[ \AZ_{bad}(\mfr{m}_0)= \{ [d-r,d], [-d,-d+r] \},\]
    which corresponds to $\AZ(\lambda)\in \Phi_{\lambda}(G_n)$.
    Since $\{\phi' \in \Phi_{\lambda_{\GL}}(\GL_{2r+2}(F)) \ | \ \phi' \gneq_C \std \circ \AZ(\lambda)\}$ is either empty or a singleton consisting of a tempered $L$-parameter that does not factor through $G_n^{\vee}$, we see that $\AZ(\lambda)$ corresponds to the unique open orbit. Hence, $\AZ(\lambda)$ matches the Pyasetskii involution of $\lambda$ in this case. 
\end{exmp}

\subsection{\texorpdfstring{Notations for proving Theorem \ref{thm bad}}{}}\label{sec setup}
In this subsection, we set up the notation used in \cite{MW86} with some modifications for classical groups. By Proposition \ref{prop unramification}, we may assume that $\phi$ is unramified, i.e., we may assume $\rho$ is trivial.
We shall also simply identify the segment $\{\lvert\cdot\rvert^{b},\ldots,\lvert\cdot\rvert^{e}\}$ with the set of half integers $\{b,\ldots,e\}$.

Thus, we let $\phi= \bigoplus_{i \in I} \lvert\cdot \rvert^{x_i} \otimes S_{a_i}$
be an $L$-parameter of $\RO_{2n}$  or $\SO_{2n+1}$ of bad parity, and write the corresponding multi-segment as
\[ \mfr{m}= \{\Delta_i=[b_i, e_i]\}_{i \in I}.\]
In the computation, we set $I(a):= \{i \in I\ | \ a \in \Delta_i\}.$ We consider the (even-dimensional) vector space 
\[W_{\mfr{m}} := \bigoplus_{i \in I} \bigoplus_{a \in \Delta_i } \BC \Delta_{i}(a)\]
with a fixed basis $\{\Delta_{i}(a)\}_{i \in I, a \in \Delta_i}$. We simply write $W_{\mfr{m}}=W$ if $\mfr{m}$ is clear from the context. We define a standard inner product $(\cdot, \cdot)$ according to this basis:
\[ (\Delta_i (a), \Delta_j(b))= \begin{cases}
    1 & \text{ if }i=j \text{ and }a=b,\\
    0& \text{ otherwise.}
\end{cases}\]
Given any $g\in\End(W)$ , we define its transpose via
$(\cdot, \cdot)$ by $(v,g(w))= ({}^tg(v),w)$ for any $v,w \in W$.

We say $g \in \End(W)$ is of degree $k$ if 
\begin{align}\label{eq degree k}
     g(\Delta_i(a))= \sum_{j \in I,\ a+k\in \Delta_j} \gamma_{i,j}(a) \Delta_j(a+k)
\end{align}
for some collection of complex numbers $\{\gamma_{i,j}(a)\}$. Let $\End(W)^{\deg=k}$ denote the collection of endomorphisms of $W$ of degree $k$.
If $g$ is of degree $k$ as in \eqref{eq degree k}, then we have
\[ {}^t g(\Delta_i(a))= \sum_{j \in I,\ a-k\in \Delta_j} \gamma_{j,i}(a-k) \Delta_j(a-k), \]
which is of degree $-k$.

We fix the involution $i \mapsto i^{\vee}$ on $I$ as in Definition \ref{def Mseg bad} throughout the rest of the section. Decompose $I= I_{>0}  \sqcup I_{=0} \sqcup I_{<0}$, where, for $\oldbullet\in\{>,=,<\}$,
\[I_{\oldbullet 0}:=\{i \in I \ | \ b_i+e_i \oldbullet 0\}.\]
We further fix a decomposition of $I_{=0}= I_{=0}^+ \sqcup I_{=0}^-$ with $(I_{=0}^+)^{\vee}= I_{=0}^-$. Then we define $I^+:= I_{>0} \sqcup I_{=0}^+$ and $I^-:= I_{<0} \sqcup I_{=0}^-$ so that $I= I^+ \sqcup I^-$ and $(I^+)^{\vee}= I^-$.

Let $\delta\in \{0, \half{1}\}$ such that $b(\Delta_i)-\delta\in \Z$ for any $i \in I$. Set $\epsilon=(-1)^{2\delta+1}$. Namely,
$\epsilon=1$ if $\delta= \half{1}$ and $\epsilon=-1$ if $\delta=0$. We define a non-degenerate $\epsilon$-symmetric bilinear form $\langle \cdot, \cdot \rangle_{\mfr{m}}$, or simply $\langle \cdot, \cdot \rangle$ if $\mfr{m}$ is clear from context, by
\[ \langle \Delta_i(a), \Delta_{j}(b)\rangle =\begin{cases}
    \epsilon(i) (-1)^{a-\delta} & \text{ if }j=i^{\vee} \text{ and }b=-a,\\
    0 & \text{ otherwise,}
\end{cases}\]
where 
\[ \epsilon(i)=\begin{cases}
    1 &\text{ if }i \in I^+,\\
    -1 &\text{ if }i \in I^-.
\end{cases}\]
With the inner product $(\cdot, \cdot)$, the bilinear form $\langle \cdot,\cdot \rangle$ is represented by $J\in \End(W)$ via
\[ \langle v,w\rangle= (v,J(w)). \]
More explicitly, we have
\begin{align}\label{eq def J}
J(\Delta_i(a))=\epsilon(i^{\vee}) (-1)^{-a-\delta} \Delta_{i^{\vee}}(-a). 
\end{align}
Note that ${}^tJ= \epsilon J=J^{-1}$.

\begin{remark}
    The choice of $\epsilon\in \{\pm 1\}$ we made is the only choice such that $\phi$ is of bad parity. It is straightforward to check that if we require that $\epsilon(i)=1$ for all $i \in I^+$, then the bilinear form is $\epsilon$-symmetric if and only if (note that $(-1)^{a-\delta}= - (-1)^{-a-\delta}$ if $a \in \half{1}+\Z$)
    \begin{align*}
    \epsilon(i)=\begin{cases}
        1 &\text{ if }i \in I^+,\\
        \epsilon & \text{ if } i \in I^{-} \text{ and }\delta\neq \half{1},\\
        -\epsilon & \text{ if } i \in I^{-} \text{ and }\delta=\half{1},
    \end{cases}
\end{align*}
    which matches our definition of $\epsilon(i)$ above.
\end{remark}

Now we rephrase the definition of $V_{\lambda}, V_{\lambda}^{\ast} $ and $H_{\lambda}$.
\begin{align*}
    V_{\lambda}&=\{g \in \End(W)^{\deg=1}\ | \ {}^tg J +Jg=0\},\\
    V_{\lambda}^{\ast}&=\{g \in \End(W)^{\deg=-1}\ | \ {}^tg J +Jg=0\},\\
      H_{\lambda}&=\{g \in \End(W)^{\deg=0}\ | \ {}^tg J g = J,\ g \text{ is invertible}\},
\end{align*}

Finally, we fix an $f \in \End(W)$ of degree $1$ defined by
\begin{align}\label{eq def f}
    f(\Delta_{i}(a)):= \begin{cases}
    \Delta_{i}(a+1) & \text{ if }a+1 \in \Delta_i,\\
    0 &\text{ otherwise.}
\end{cases}
\end{align}
Indeed, we have $f$ lies in $C_{\phi} \subseteq V_{\lambda}$, as shown in the following lemma. 
\begin{lemma}\label{lem f in C phi}
    We have ${}^t f J+Jf=0$. Moreover, $f \in C_{\phi}$.
\end{lemma}
\begin{proof} If $b \not\in \Delta_{j}$, then we set $\Delta_{j}(b)=0$. Then
\begin{align*}
    {}^t f J (\Delta_i(a))&= {}^t f( \epsilon(i^{\vee})(-1)^{{-}a-\delta} \Delta_{i^{\vee}}(-a) )\\
    &=\epsilon(i^{\vee})(-1)^{{-}a-\delta} \Delta_{i^{\vee}}(-a-1)\\
    &= -\epsilon(i^{\vee})(-1)^{-a-1-\delta} \Delta_{i^{\vee}}(-a-1)\\
    &=-J( \Delta_i(a+1) ) \\
    &= -J f(\Delta_i(a)).
\end{align*}
Thus, $f \in V_{\lambda}$. Finally, it follows from the definition that the rank matrix of $f$ is equal to the rank matrix of $\phi$. Thus, $f \in C_{\phi}.$
\end{proof}

Finally, we denote the commutant of $f$ by
\[ C(f):= \{ g \in \End(W)\ | \ fg=gf\}.\]

\subsection{\texorpdfstring{Commutant of $f$ in  $V_{\lambda}^{\ast}$ and $H_{\lambda}$}{}}
In this subsection, we generalize \cite[Lemmas II.4, II.5]{MW86} in our setting.

\begin{lemma}\label{lem fg=gf Lie algebra}
    Let $g \in \End(W)^{\deg=-1}$. Write
    \[ g(\Delta_{i}(a))= \sum_{j \in I(a-1)} \alpha_{i,j}(a) \Delta_{j}(a-1),\]
   and set $\alpha_{i,j}(a)=0$ if $a-1 \not\in \Delta_j$ or $a \not\in \Delta_i$.
    \begin{enumerate}
        \item [(a)] We have $g \in C(f)$ if and only if
        \begin{enumerate}
            \item [(i)] $\alpha_{i,j}(a)=0$ if $\Delta_j $ does not precede $ \Delta_i$, and
            \item [(ii)] $\alpha_{i,j}(a)$ is independent of $a$ as long as $a \in \Delta_i$ and $a-1 \in \Delta_j$.
        \end{enumerate}
        Thus, we let $\alpha_{i,j}(g):= \alpha_{i,j}(a)$ for any $a \in \Delta_i$ with $a-1 \in \Delta_j$.
        \item [(b)]  We have ${}^tgJ+Jg=0$ if and only if 
        \begin{itemize}
            \item [(iii)] $\alpha_{i,j}(a)= \epsilon(i^{\vee})\epsilon(j^{\vee}) \alpha_{j^{\vee},i^{\vee}}(-a+1).$
        \end{itemize}
    \end{enumerate}
\end{lemma}
\begin{proof}
    Part (a) is exactly \cite[Lemma II.4]{MW86}, which we omit. For Part (b), we have
    \[ {}^t g J (\Delta_i(a))= {}^t g( \epsilon(i^{\vee})(-1)^{-a-\delta} \Delta_{i^{\vee}}(-a) )= \sum_{j^{\vee} \in I(-a+1)} \epsilon(i^{\vee})(-1)^{-a-\delta} \alpha_{j^{\vee},i^{\vee}}(-a+1)\Delta_{j^{\vee}}(-a+1), \]
    and
    \[ Jg( \Delta_i(a))=J(\sum_{j \in I(a-1)} \alpha_{i,j}(a)\Delta_{j}(a-1))=\sum_{j \in I(a-1)} \epsilon(j^{\vee}) (-1)^{-a+1-\delta}\alpha_{i,j}(a) \Delta_{j^{\vee}}(-a+1).\]
    Note that $I(-a+1)^{\vee}=I(a-1)$. Thus, comparing two sides, we see that ${}^t gJ+Jg=0$ if and only if Condition (iii) holds. This completes the proof of the lemma.    
\end{proof}

We also record the analogous statement for $H_{\lambda} \cap C(f).$

\begin{lemma}\label{lem fg=gf Lie group}
    Let $h \in \End(W)^{\deg=0}$. Write
    \[ h(\Delta_{i}(a))= \sum_{j \in I(a)} \beta_{i,j}(a) \Delta_{j}(a).\]
  \begin{enumerate}
      \item [(a)] We have $h \in C(f)$ if and only if  
      \begin{enumerate}
            \item [(i)] $\beta_{i,j}(a)=0$ if $^+(\Delta_j)^- $ does not precede $ \Delta_i$, and
            \item [(ii)] $\beta_{i,j}(a)$ is independent of $a$ as long as $a \in \Delta_i$ and $a \in \Delta_{j}$.
        \end{enumerate}
        Thus, we let $\beta_{i,j}:= \beta_{i,j}(a)$ for any $a \in \Delta_i$ with $a \in \Delta_j$. 
        \item [(b)] We have $h \in H_{\lambda}$ if and only if $h$ is invertible and ${}^thJh=J$.
  \end{enumerate}
\end{lemma}
\begin{proof}
    Part (a) is exactly \cite[Lemma II.5]{MW86}, which we omit. Part (b) is by definition.
\end{proof}

\subsection{A technical proposition}
 Let $\{0,1,\ldots,r\} \subseteq I$ be the indices obtained from applying Steps 2 and 3 of Algorithm \ref{algo LM25} to $\mfr{m}$ and write $d=e(\Delta_0)$. The goal of this subsection is to prove Proposition \ref{prop g r+1=0}, which is crucial in the proof of Theorem \ref{thm bad}.
We first develop some notations.

\begin{defn}Fix a $g \in V_{\lambda}^{\ast} \cap C(f)$.
    \begin{enumerate}
        \item[(a)]  For $i,j \in I$, we define $\psi_{i,j}^s(g) \in \BC$ by
        \[ g^s(\Delta_i(e(\Delta_i)))= \sum_{j\in I(e(\Delta_i)-s)} \psi_{i,j}^s(g) \Delta_j(e(\Delta_i)-s).\]
        We set $\psi_{i,j}^s(g)=0$ if $e(\Delta_i)-s \not\in \Delta_j$.
        \item [(b)] For $i,j \in I$ with $s:=e(\Delta_i)-e(\Delta_j)$, if $s>0$, we define
        \[ \Omega_{i,j}:=\{(i,i_1,\ldots,i_{s-1},j)\ | \ \Delta_i \succ \Delta_{i_1}\succ \cdots \succ \Delta_{i_{s-1}}\succ \Delta_j\}.\]
        Otherwise if $s \leq 0$, we set $\Omega_{i,j}=\emptyset$. For any $(i,i_1,\ldots,i_{s-1},j) \in \Omega_{i,j}$, we define
        \[ \alpha_{(i,i_1,\ldots,i_{s-1},j)}(g):= \alpha_{i,i_1}(g) \alpha_{i_1, i_2}(g) \cdots \alpha_{i_{s-1},j}(g),\]
        where $\alpha_{i_{k},i_{k+1}}(g)$ is the complex number associated to $g$ as in Lemma \ref{lem fg=gf Lie algebra}(a).
     \end{enumerate}
\end{defn}

\begin{remark}\label{rmk Omega psi}\ 
\begin{enumerate}
    \item  If $(i,i_1,\ldots,i_{s-1},j) \in \Omega_{i,j}$, then $e(\Delta_{i_k})= e(\Delta_{i})-k$ for $1\leq k\leq s-1$.
    \item By Lemma \ref{lem fg=gf Lie algebra}(a), for $i,j \in I$ with $s:=e(\Delta_i)-e(\Delta_j)$, we have
    \[ \psi^{s}_{i,j}(g)= \sum_{ (i,i_1,\ldots,i_{s-1},j) \in \Omega_{i,j}} \alpha_{(i,i_1,\ldots,i_{s-1},j)}(g), \]
     and $\psi_{i,j}^t(g)=0$ if $t \neq s$.
\end{enumerate}
\end{remark}

\begin{lemma}\label{lem i i vee}
    Let $i\in I^+$ and $s:=e(\Delta_i)-e(\Delta_{i^{\vee}})$. We have $\psi_{i,i^{\vee}}^s(g)=0$ for any $g \in V_{\lambda}^{\ast} \cap C(f)$.
\end{lemma}
\begin{proof}
    If $ \Omega_{i,i^{\vee}}=\emptyset$, the lemma holds obviously by Remark \ref{rmk Omega psi}. Thus, we assume $\Omega_{i,i^{\vee}}$ is non-empty in the rest of the proof. Take any $\iota:=(i, i_1,\ldots, i_{r-1}, i^{\vee})\in \Omega_{i,i^{\vee}}$. The precedence condition and $e(\Delta_{i_k})= e(\Delta_{i})-k$ imply that
    \[ l(\Delta_i) \leq l(\Delta_{i_1})\leq \cdots \leq l(\Delta_{i_{r-1}}) \leq l(\Delta_{i^{\vee}}).\]
    Since $l(\Delta_{i^{\vee}})=l((\Delta_i)^{\vee})=l(\Delta_i)$, we see that the inequalities above are all equalities. Hence, $\Delta_{i_s}= [ b(\Delta_i)-s, e(\Delta_i)-s]$ for $1 \leq s \leq r-1$. Then it is not hard to check that the sequence of indices $\iota^{\vee}:=(i, i_{r-1}^\vee,\ldots, i_1^{\vee}, i^{\vee} )$ is also in $\Omega_{i,i^{\vee}}$ as well. 

    Now take any $\iota=(i, i_1,\ldots, i_{r-1}, i^{\vee})\in \Omega_{i,i^{\vee}}$ with $i, i_1,\ldots ,i_s \in I^+$ and $i_{s+1},\ldots, i_{r-1}, i^{\vee} \in I^{-}$. Suppose that $\iota=\iota^{\vee}$. Then $i_{s}^{\vee}= i_{s+1}$. Hence, Lemma \ref{lem fg=gf Lie algebra}(b) implies that for any $g \in V_{\lambda}^{\ast} \cap C(f)$,
    \[\alpha_{i_{s},i_{s+1}}(g)=- \alpha_{i_{s+1}^{\vee},i_{s}^{\vee}}(g)=- \alpha_{i_{s},i_{s+1}}(g). \]
    Hence, $\alpha_{i_{s},i_{s+1}}(g)=0$ and $\alpha_{\iota }(g)=0$. 
    
    Suppose $\iota\neq \iota^{\vee}$. Then similarly, Lemma \ref{lem fg=gf Lie algebra}(b) implies that 
    \begin{align*}
        \alpha_{\iota}(g)&= (\alpha_{i,i_1}(g)\cdots \alpha_{i_{s-1},i_s}(g)) \alpha_{i_s,i_{s+1}}(g)  (\alpha_{i_{s+1},i_{s+2}}(g)\cdots \alpha_{i_{r-1},i^{\vee}}(g))\\
        &= (\alpha_{i_1^{\vee},i^{\vee}}(g)\cdots \alpha_{i_{s}^{\vee},i_{s-1}^{\vee}}(g)) (- \alpha_{i_{s+1}^{\vee},i_{s}^{\vee}}(g))  (\alpha_{i_{s+2}^{\vee},i_{s+1}^{\vee}}(g)\cdots \alpha_{i,i_{r-1}^{\vee}}(g))\\
        &=- \alpha_{i, i_{r-1}^{\vee}}(g) \cdots \alpha_{i_{s+2}^{\vee},i_{s+1}^{\vee}}(g) \alpha_{i_{s+1}^{\vee},i_{s}^{\vee}}(g) \alpha_{i_{s}^{\vee},i_{s-1}^{\vee}}(g) \cdots \alpha_{i_{1}^{\vee},i^{\vee}}(g)\\
        &= -\alpha_{\iota^{\vee}}(g).
    \end{align*}
Therefore,
\begin{align*}
    \psi_{i,i^{\vee}}^s(g)&= \sum_{\iota \in \Omega_{i,i^{\vee}}}  \alpha_{\iota}(g)\\
    &= \sum_{\substack{\iota \in \Omega_{i,i^{\vee}},\ \iota =\iota^{\vee} }}  \alpha_{\iota}(g) 
    + \sum_{\substack{\{\iota,\iota^\vee \} \subset \Omega_{i,i^{\vee}}^2,\ \iota \neq \iota^{\vee}}} \alpha_{\iota}(g)+\alpha_{\iota^{\vee}}(g)\\
    &=0.
\end{align*}
    This completes the proof of the lemma.
\end{proof}

We need one more lemma as an analogue of \cite[Remark II.2.1]{MW86}.

\begin{lemma}\label{lem r+1}
     Let $\{0,1,\ldots,r\} \subseteq I$ be the indices obtained from applying Steps 2,3 of Algorithm \ref{algo LM25} to $(\mfr{m},\vee) \in \MSeg_{bad}$ and write $d=e(\Delta_0)$. Suppose $(i_0, \ldots, i_{s})$ is a sequence of indices such that
     \begin{itemize}
         \item $e(\Delta_{i_k})= d-k$ for $0 \leq k \leq r+1$, and 
         \item $\Delta_{i_0} \succ \cdots \succ \Delta_{i_{s}}$.
     \end{itemize}
     Suppose that $l(\Delta_k)>l(\Delta_{i_k})$ for some $1 \leq k \leq \min(s,r+1)$, where we formally set $l(\Delta_{r+1}):= +\infty$.
  Then there exists an index $i$ such that $\{i,i^{\vee}\} \subseteq \{i_0,\ldots, i_{s}\}$.
\end{lemma}
\begin{proof}
Note that $l(\Delta_0)\leq l(\Delta_{i_0})$ by Step 2 of Algorithm \ref{algo LM25}. We may assume $l(\Delta_k) \leq l(\Delta_{i_k})$ for $0 \leq k \leq s-1$ and $l(\Delta_{s})> l(\Delta_{i_s})$. 

We first show that for $0 \leq k \leq s-1$, the sequence $( 0,\ldots, k, i_{k+1},\ldots, i_{s})$ lies in $\Omega_{0,i_{s}}$. Indeed, this follows directly from the observation that if 
  $\Delta_{j_1}\prec \Delta_{j_2}$ and $j_2'$ is another index such that  $l(\Delta_{j_2'})\leq l(\Delta_{j_2})$ and $e(\Delta_{j_2'})=e(\Delta_{j_2})=e(\Delta_{j_1})+1$, then $\Delta_{j_1} \prec \Delta_{j_2'}$ as well.

Next, we verify that $i_s^{\vee}\in \{i_0 ,\ldots, i_{s-1}\}.$ By Step 3 of Algorithm \ref{algo LM25}, $i_s$ cannot lie in the set $I_s$ defined there. Since $l(\Delta_{s-1}) \leq l(\Delta_{i_{s-1}})$, we have
 $\Delta_{i_s}\prec \Delta_{s-1}$. Hence, the index $i_s$ satisfies Condition (a) in Step 3 of Algorithm \ref{algo LM25}. Therefore, Condition (b) must fail for $i_s$, and we conclude that $i_s^{\vee}=j \in \{0,\ldots, s-1\}$ and $ \{i \in I \ | \ \Delta_i=\Delta_{i_s} \}=\{i_s\}$, which is related to the following set by taking the involution $\vee$
 \begin{align}\label{eq lemma seq}
     \{i \in I \ | \ \Delta_i=\Delta_{j} \}=\{j\}.
 \end{align} On the other hand, we have
\[ l(\Delta_j) \leq l(\Delta_{i_j}) \leq l(\Delta_{i_{j+1}}) \leq \cdots \leq l(\Delta_{i_s})= l(\Delta_{j^{\vee}}).\]
Since $l(\Delta_{j^{\vee}})=l(\Delta_{j})$, the above inequalities are all equalities. Now $e(\Delta_{i_j})=d-j=e(\Delta_j)$ and $l(\Delta_j) = l(\Delta_{i_j}) $ imply that $i_s^{\vee}=j=i_j$ by \eqref{eq lemma seq}. This completes the proof of the lemma.
\end{proof}

    

Now we state and prove the main proposition of this subsection. Note that in the statement below, $\psi^s_{0,j}(g) \neq 0$ implies that $e(\Delta_j)=d-s$. In this case, $^+(\Delta_j)^- $ precedes $ \Delta_s$ if and only if $l(\Delta_j) \geq l(\Delta_s)$.

\begin{prop}\label{prop g r+1=0}
Let $\{0,1,\ldots,r\} \subseteq I$ be the indices obtained from applying Steps 2,3 of Algorithm \ref{algo LM25} to $(\mfr{m},\vee) \in \MSeg_{bad}$ and write $d=e(\Delta_0)$. For any $g \in V_{\lambda}^{\ast} \cap C(f)$, the following holds.
\begin{enumerate}
    \item [(a)]We have $g^{r+1}(\Delta_{0}(d))=0.$
    \item [(b)] If $\psi^{s}_{0,j}(g) \neq 0$ for some $1\leq s \leq r$, then $^+(\Delta_j)^- $ precedes $ \Delta_s$.
\end{enumerate}
\end{prop}
\begin{proof}
  In this proof, if $\iota_1=(i_0,\ldots, i_s)$ and $\iota_2=(i_{s+1},\ldots, i_t)$ are two sequences of indices, then we define $(\iota_1,\iota_2):=(i_0,\ldots, i_s, i_{s+1},\ldots, i_t)$.
By Remark \ref{rmk Omega psi},
it suffices to show that for any $g \in V_{\lambda}^{\ast} \cap C(f)$, and $j \in I$ such that
\begin{enumerate}
    \item [(a)] $e(\Delta_j)=d-r-1$, or
    \item [(b)] $e(\Delta_j)= d-s$ but    
    $l(\Delta_j)<l(\Delta_{s})$,
\end{enumerate}
we must have
\[ \sum_{\iota \in \Omega_{0,j}} \alpha_{\iota}(g)=0.\]

For any $k \in I^+$, we define a (possibly empty) set
\[ \Xi_{0,j}(k):= \{((0,i_1,\ldots, i_s, k), (k^{\vee},i_{s+1},\ldots, i_t,j) )\in \Omega_{0,k}\times \Omega_{k^{\vee},j} \ | \ \{0^{\vee}, i_1^{\vee},\ldots, i_s^{\vee}\}\cap (i_{s+1},\ldots, i_t,j)=\emptyset  \}. \]
Then for each $(\iota_1,\iota_2)\in \Xi_{0,j}(k)$, we define
\[ \Omega_{\iota_1,\iota_2}:= \{ (\iota_1, \iota' , \iota_2 ) \in \Omega_{0,j} \ | \ (k,\iota', k^{\vee}) \in \Omega_{k,k^{\vee}} \}.\]
Note that this set is in bijection with $\Omega_{k,k^{\vee}}$ in an obvious way. Thus, Lemma \ref{lem i i vee} implies that for any $g \in V^{\ast} \cap C(f)$, we have ($s:=e(\Delta_k)-e(\Delta_{k^{\vee}})$)
\begin{align}\label{eq alpha iota 1 iota 2=0}
    \sum_{\iota \in \Omega_{\iota_1, \iota_2}} \alpha_{\iota}(g)= \alpha_{\iota_1}(g)\psi_{k,k^{\vee}}^s \alpha_{\iota_2}(g)=0.
\end{align}

Finally, we have
\[ \Omega_{0,j}=\bigsqcup_{k\in I^+} \bigsqcup_{(\iota_1,\iota_2)\in \Xi_{0,j}(k)} \{ \Omega_{\iota_1,\iota_2} \}.\]
Indeed, the definition of $\Xi_{0,j}(k)$ guarantees that the right hand side is a disjoint union of subsets of $\Omega_{i,j}$.
On the other hand, Lemma \ref{lem r+1} implies that the left hand side is contained in the right hand side. Then as a consequence of \eqref{eq alpha iota 1 iota 2=0}, we obtain that 
\[\sum_{\iota \in \Omega_{0,j}} \alpha_{\iota}(g)= \sum_{k\in I^+} \sum_{(\iota_1,\iota_2)\in \Xi_{0,j}(k)} \sum_{\iota \in \Omega_{\iota_1,\iota_2}} \alpha_{\iota}(g)=0.\]
This completes the proof of the proposition.
\end{proof}

\subsection{\texorpdfstring{Preparations for proving Theorem \ref{thm bad}}{}}
Following the strategy of M{\oe}glin-Waldspurger, we verify analogs of Assertions (1) and (3) in the proof of \cite[Proposition II.6]{MW86} in this subsection. Let $\{0,1,\ldots,r\} \subseteq I$ be the indices obtained from applying Steps 2 and 3 of Algorithm \ref{algo LM25} to $\mfr{m}$ and write $d=e(\Delta_0)$. According to Step 3 of Algorithm \ref{algo LM25}, we may assume that 
\[\{0,\ldots, r\} \cap \{0^{\vee},\ldots, r^{\vee}\}= \emptyset\]
by replacing certain $i$ with $i' \neq i$ such that $\Delta_i=\Delta_{i'}$ if necessary.

Let $U:=U_1\oplus U_2$ where
\begin{align*}
    U_1&:= \bigoplus_{i=0}^r \BC \Delta_i(d-i), \\
     U_2&:= \bigoplus_{i=0}^r   \BC \Delta_{i^{\vee}}(i-d), 
\end{align*}
and let $W^-$ be the orthogonal complement of $U$ with respect to the inner product $(\cdot,\cdot)$, so that $W=W^- \oplus U$.
It is clear from the definition that $W^-\cong W_{\mfr{m}^{\#}}$, the vector space associated to $\mfr{m}^{\#}$ defined in Step 4 of Algorithm \ref{algo LM25}. Moreover, the $\epsilon$-symmetric bilinear form $\langle \cdot, \cdot \rangle_{\mfr{m}^{\#}}$ on $W^-$ defined by the same recipe in \S \ref{sec setup} is exactly the restriction of $\langle \cdot, \cdot \rangle_{\mfr{m}}$.

Let $\mathcal{W}_1, \mathcal{W}_2$ be two vector subspaces of $W$, we let $\iota_{\mathcal{W}_1}: \mathcal{W}_1 \hookrightarrow W$ be the inclusion and $p_{\mathcal{W}_2}: W \to \mathcal{W}_2$ be the orthogonal projection with respect to $(\cdot,\cdot)$. For any endomorphism $h \in \End(W)$, we define $h_{\mathcal{W}_1,\mathcal{W}_2} \in \Hom_{\BC}(\mathcal{W}_1, \mathcal{W}_2)$ by
\[ h_{\mathcal{W}_1,\mathcal{W}_2}:= p_{\mathcal{W}_2} \circ h \circ \iota_{\mathcal{W}_1}.\]
We shall write $h_{\mathcal{W}_1}:=h_{\mathcal{W}_1,\mathcal{W}_1}$ for short. Also, we write 
\[ h_{\mathcal{W}_1,W}= h_{\mathcal{W}_1,U_1}\oplus h_{\mathcal{W}_1,W^{-}} \oplus h_{\mathcal{W}_1,U_2}, \ h_{W,\mathcal{W}_2}= h_{U_1,\mathcal{W}_2}\oplus h_{W^-,\mathcal{W}_2} \oplus h_{U_2,\mathcal{W}_2}. \]


\begin{lemma}\label{lem first conjugation}
    There exists a $g \in C_{\widehat{\phi}}^{\ast}\cap C(f)$ such that for $0 \leq k \leq r$,
    \[g(\Delta_k(d-k)) =\begin{cases}
        \Delta_{k+1}(d-k-1) &\text{ if }0 \leq k \leq r-1,\\
        0& \text{ if }k=r.
    \end{cases}\]
    In particular, $g(U_1) \subset U_1$.
\end{lemma}
\begin{proof}
    First, we show that there exists a $g' \in C_{\widehat{\phi}}^{\ast} \cap C(f)$ such that 
    \begin{enumerate}
        \item [(1)] $(g')^{r+1}(\Delta_{0}(d))=0$, and
        \item [(2)] $\psi^{s}_{0,s}(g') \neq 0$ for $1 \leq s \leq r$, and 
        \item [(3)] if $\psi^{s}_{0,j}(g') \neq 0$, then $^+(\Delta_j)^- $ precedes $ \Delta_s$.
    \end{enumerate}
    Indeed, Conditions (1) and (3) hold for every $g' \in C(f) \cap V_{\lambda}^{\ast}$ by Proposition \ref{prop g r+1=0}. For Condition (2), since $C_{\widehat{\phi}}^{\ast} \cap C(f)$ is dense in $C(f)$ by definition (see Remark \ref{rmk Pyasetskii involution}), it suffices to show that there exists an $g'' \in C(f)\cap V_{\lambda}^{\ast}$ such that $\psi^{s}_{0,s}(g'') \neq 0$ for $1 \leq s \leq r$. Indeed, we can take $g'' \in \End(W)$ defined by
    \begin{align*}
        g'' (\Delta_{i}(a)):= \begin{cases}
            \Delta_{i+1}(a-1) &\text{ if }i \in \{0,\ldots, r-1\},\\
            \epsilon(j^{\vee})\epsilon((j-1)^{\vee}) \Delta_{(j-1)^{\vee}}(a-1) &\text{ if }i^{\vee}=j \in \{1,\ldots, r\},\\
            0&\text{ otherwise}.
        \end{cases}
    \end{align*}

Next, we construct an $h \in \End(W)^{\deg=0}$ as follows. 
{Let $U_1'=\bigoplus_{i=0}^r\bigoplus_{a\in\Delta_i}\BC\Delta_i(a).$ We define $h_{U_1',W} \in \Hom_{\BC}(U_1',W)$ by
\begin{align}\label{eq h U1' W}
    h_{U_1', W}(\Delta_{i}(a)):= \begin{cases}
\Delta_0(a) &\text{ if }i\in \{0\},\\
 \sum_{j}\psi_{0,j}^{i}(g')\Delta_j(a)&\text{ if }i \in \{1,\ldots, r\}. \end{cases}
\end{align}
As a special case of the above formula, for $i\in \{0,\ldots, r\}$, we have
\begin{align}\label{eq design h}
    h_{U_1',W}(\Delta_i(d-i))= (g')^{i}(\Delta_0(d)).
\end{align}
Also, for $j \in \{1,\ldots, r\}$, we have $\psi_{0,j}^i(g')=0$ unless $j=i$ by Remark \ref{rmk Omega psi}(2). In other words, 
\begin{equation}\label{eqn h U1 diagonal}
h_{U_1'}(\Delta_{i}(a))=\psi_{0,i}^i(g')\Delta_{i}(a)    
\end{equation}
for $i \in \{1,\ldots, r\}$. Note that $\psi_{0,i}^i(g')\neq 0$ for $i \in \{1,\ldots, r\}$ by Condition (2) for $g'$ above.
}


{
Let $U'=U_1'\oplus U_2'$ where $U_2'=\bigoplus_{i=0}^r\bigoplus_{a\in\Delta_{i^\vee}}\BC\Delta_{i^\vee}(a)=\bigoplus_{i=0}^r\bigoplus_{a\in\Delta_{i}}\BC\Delta_{i^\vee}(a-d)$ and $W'$ denote the orthogonal complement of $U'$ in $W$ with respect to $(\cdot,\cdot),$ i.e., $W=U'\oplus W'.$
We claim that there is a unique extension $h\in \End(W)^{\deg=0}$ of $h_{U_1',W}$ such that 
\begin{itemize}
    \item[(a)] $h_{W'}=\id_{W'}$, $h_{W', U_1'}=0$, $h_{U_2',U_1'}=0$ and $h_{U_2',W'}=0$, and
    \item[(b)] ${}^t h J h=J$, and 
    \item[(c)] $fh=hf$.
\end{itemize}

First, we show the uniqueness and existence of an extension satisfying Conditions (a) and (b). We begin with uniqueness. Since $J(U_1')=U_2'$, $J(W')=W'$ and $J(U_2')=U_1'$, we may write
\[ J=J_{U_1',U_2'}\oplus J_{W'} \oplus J_{U_2',U_1'}\]
where $J_{U_1',U_2'}$, $J_{W'}$ and $J_{U_2',U_1'}$ are invertible. Condition (a) implies that 
\[ h= h_{U_1'} \oplus h_{U_1',W'} \oplus h_{U_1',U_2'} \oplus \id_{W'} \oplus h_{W', U_2'}\oplus h_{U_2'}, \]
where $h_{U_1'} \oplus h_{U_1',W'} \oplus h_{U_1',U_2'}= h_{U_1', W}$ is given in \eqref{eq h U1' W}.
A direct computation translates Condition (b) into the following equations (recall that {${}^tJ= \epsilon J$}).
\begin{align*}
    \begin{cases}
        {}^t h_{U_1',U_2'} \circ J_{U_1',U_2'} \circ h_{U_1'}+ {}^th_{U_1',W'}\circ J_{W'}\circ  h_{U_1',W'}+ {}^t h_{U_1'} \circ J_{U_2',U_1'}\circ h_{U_1',U_2'}&=0,\\
        {}^t h_{U_1',W'} \circ J_{W'}+ {}^th_{U_1'} \circ J_{U_2',U_1'} \circ h_{W',U_2'}&=0,\\
        {}^th_{U_1'} \circ J_{U_2',U_1'}\circ h_{U_2'}&= J_{U_2',U_1'}.
    \end{cases}
\end{align*}
Since $h_{U_1'}$ and $J_{U_2',U_1'}$ are invertible, the second and the third equations imply that
\begin{align*}
    \begin{cases}
    h_{W',U_2'}&= - J_{U_2',U_1'}^{-1} \circ {}^th_{U_1'}^{-1} \circ {}^t h_{U_1',W'} \circ J_{W'},\\ 
        h_{U_2'}&=  J_{U_2',U_1'}^{-1} \circ {}^th_{U_1'}^{-1} \circ J_{U_2',U_1'}.
    \end{cases}
\end{align*}
This shows the uniqueness of the extension with Conditions (a) and (b). 

For the existence, we verify the first equation. Note that it is equivalent to ${}^t h_{U_1',W} \circ J \circ h_{U_1',W}=0$, which is also equivalent to
\[ \langle h_{U_1',W}(\Delta_{i}(d-i)), h_{U_1',W}(\Delta_j(d-j))\rangle=0,\]
for any $i,j \in \{0,\ldots, r\}.$ Now by \eqref{eq design h} and the condition that ${}^t g'J=-J g'$, we have 
\begin{align*}
     \langle h_{U_1',W}(\Delta_{i}(d-i)), h_{U_1',W}(\Delta_j(d-j))\rangle&= \langle (g')^i (\Delta_0(d)), (g')^j (\Delta_0(d))\rangle\\
     &= ( (g')^i (\Delta_0(d)), J (g')^j (\Delta_0(d)) )\\
     &= ( \Delta_0(d), {}^t (g')^i J (g')^j (\Delta_0(d)) )\\
     &= (-1)^i ( \Delta_0(d),  J (g')^{i+j} (\Delta_0(d)) )\\
     &=(-1)^i \langle \Delta_0(d), (g')^{i+j} (\Delta_0(d))\rangle\\
     &=\begin{cases}
       (-1)^{d-\delta+i}  \psi_{0,0^{\vee}}^{i+j}(g') & \text{ if }i+j=2d,\\
       0& \text{ otherwise,}
     \end{cases}\\
     &=0.
\end{align*}
Here, the last equality follows from Remark \ref{rmk Omega psi}(2) and Lemma \ref{lem i i vee}.
}

{
Next, we show that this unique extension satisfies Condition (c).
Write \[ h(\Delta_{i}(a))= \sum_{j \in I(a)} \beta_{i,j}(a) \Delta_{j}(a).\]
By Lemma \ref{lem fg=gf Lie group}(1), it suffices to show that 
\begin{enumerate}
    \item [(i)] $\beta_{i,j}(a)=0$ if $^+(\Delta_j)^- $ does not precede $ \Delta_i$, and
            \item [(ii)] $\beta_{i,j}(a)$ is independent of $a$ as long as $a \in \Delta_i$ and $a \in \Delta_{j}$.
\end{enumerate}

{

We verify Conditions (i) and (ii) simultaneously. 
For $i\in\{0,\dots,r\}$ and any $j$, Conditions (i) and (ii) follows directly from Equation \eqref{eq h U1' W} and Proposition \ref{prop g r+1=0} that if $\psi^{i}_{0,j}(g') \neq 0$, then $^+(\Delta_j)^- $ precedes $ \Delta_i$.

Recall from Equation \eqref{eqn h U1 diagonal} that $h_{U_1'}(\Delta_i(a))=\psi_{0,i}^i(g')\Delta_i(a)$ for $\Delta_i(a)\in U_1'$ and so $h_{U_1'}$ is diagonal.
For $\Delta_i(a)\in U_2'$, we have \[h(\Delta_i(a))=(J_{U_2',U_1'}^{-1} \circ {}^th_{U_1'}^{-1} \circ J_{U_2',U_1'})(\Delta_i(a))\] and so it follows that Conditions (i) and (ii) hold for $i\in\{0^\vee,\dots,r^\vee\}.$

Finally, let $\Delta_i(a)\in W'.$ Then we have $h(\Delta_i(a))=\Delta_i(a)+h_{W',U_2'}(\Delta_i(a)).$  
In particular, we obtain that $\beta_{i,j}(a)=0$ for any $j\not\in\{0^\vee,\dots,r^\vee,i\}$. For $j=i,$ we have $\beta_{i,i}(a)=1$ but  ${}^+(\Delta_i)^-$ precedes $\Delta_i$ and so  Condition (i) still holds. 
For the remaining case, we have that
\begin{align*}
    h_{W',U_2'}(\Delta_i(a))&= (- J_{U_2',U_1'}^{-1} \circ {}^th_{U_1'}^{-1} \circ {}^t h_{U_1',W'} \circ J_{W'})(\Delta_i(a))\\
    &= (- J_{U_2',U_1'}^{-1} \circ {}^th_{U_1'}^{-1} \circ {}^t h_{U_1',W'} )(\epsilon(i^\vee)(-1)^{-a-\delta}\Delta_{i^\vee}(-a))\\
    &= \epsilon(i^\vee)(-1)^{-a-\delta} (- J_{U_2',U_1'}^{-1} \circ {}^th_{U_1'}^{-1} )(\sum_{0 \leq k \leq r} \beta_{k,i^{\vee}}(-a) \Delta_{k}(-a))\\
    &= \epsilon(i^\vee)(-1)^{-a-\delta} (- J_{U_2',U_1'}^{-1} )(\sum_{0 \leq k \leq r} \beta_{k,i^{\vee}}(-a) \psi_{0,k}^k(g')\Delta_{k}(-a))\\
    &= \sum_{0 \leq k \leq r} \gamma_{i,k} \beta_{k,i^{\vee}}(-a) \Delta_{k^{\vee}}(a)
\end{align*}
for some nonzero number $\gamma_{i,k}$ not depending on $a$. In other words, for $j\in\{0^\vee,\dots,r^\vee\},$ we have 
\[ \beta_{i,j}(a)= \gamma_{i,j^{\vee}} \beta_{j^\vee,i^{\vee}}(-a).\]
Then Conditions (i) and (ii) follows from the observation that
\[ {}^+(\Delta_{j^{\vee}})^- \preceq \Delta_{i^\vee} \Longleftrightarrow {}^+(\Delta_{i})^- \preceq \Delta_{j}\]
and the conclusion for $j^\vee\in\{0,\dots,r\}.$




}


We also have that $h\in H_\lambda.$ Indeed, by Lemma \ref{lem fg=gf Lie group}(2), it suffices to show that $h$ is invertible. Recall that $h= h_{U_1'} \oplus h_{U_1',W'} \oplus h_{U_1',U_2'} \oplus \id_{W'} \oplus h_{W', U_2'}\oplus h_{U_2'}.$ Since $h_{U_1'}(\Delta_{i}(a))=\psi_{0,i}^i(g')\Delta_{i}(a),$ we obtain that $h_{U_1'}$ is invertible. This implies that $h_{U_2'}= J_{U_2',U_1'}^{-1} \circ {}^th_{U_1'}^{-1} \circ J_{U_2',U_1'}$ is invertible. Finally, since $ h_{W'}=\mathrm{id}_{W'}$ is invertible, we obtain that $h$ is also invertible.
}

{
Let $g:=h^{-1}\circ g'\circ h.$ Since $g' \in C_{\widehat{\phi}}^{\ast} \cap C(f)$ and $h\in H_\lambda\cap C(f)$, we have $g \in C_{\widehat{\phi}}^{\ast} \cap C(f).$ Recall that 
for $i\in \{0,\ldots, r\}$ we have 
\[ h(\Delta_i(d-i))= (g')^{i}(\Delta_0(d)). \]
Thus, for $i\in \{0,\ldots, r-1\}$, we obtain
\[ g (\Delta_i(d-i))= \Delta_{i+1}(d-i-1).\]
Furthermore, since $(g')^{r+1}(\Delta_0(d))=0,$ we have
\[
g (\Delta_r(d-r))= 0.
\]
Therefore, there exists a $g \in C_{\widehat{\phi}}^{\ast}\cap C(f)$ such that for $0 \leq k \leq r$,
    \[g(\Delta_k(d-k)) =\begin{cases}
        \Delta_{k+1}(d-k-1) &\text{ if }0 \leq k \leq r-1,\\
        0& \text{ if }k=r.
    \end{cases}\]
This completes the proof of the lemma.
}
\end{proof}

Recall that $W^- \cong W_{\mfr{m}^{\#}}$. Let $\phi^{\#}$ be the $L$-parameter corresponding to $\mfr{m}^{\#}$ and let $\lambda^{\#}:= \lambda_{\phi^{\#}}$.
It follows from the same proof of Lemma \ref{lem f in C phi} that $f_{W^-} \in C_{\phi^{\#}} \subseteq V_{\lambda^{\#}}$.

\begin{lemma}\label{lem g-}
    Suppose that $g$ lies in $C(f) \cap V_{\lambda}^{\ast}$ and $g(U_1) \subseteq U_1$. Then $g_{W^-}$ lies in $C(f_{W^-}) \cap V_{\lambda^{\#}}^{\ast}$
\end{lemma}
\begin{proof}
We write $p:=p_{W, W^-}$ and $\iota:= \iota_{W^-, W}$ for short in this proof. Recall the definition that
\[ g_{W^-}= p \circ g \circ \iota,\ \ f_{W^-} =p \circ f \circ \iota.\]
Since the isomorphism $W_{\mfr{m}^{\#}} \cong W^-$ preserves the $\epsilon$-symmetric bilinear form, we have $g_{W^-} \in V_{\lambda^{\#}}^{\ast}$. It remains to show that $g_{W^-}$ lies in $C(f_{W^-})$.

First, we claim that
\begin{enumerate}
    \item [(i)] $J(U_1)=J(U_2)$.
    \item [(ii)] $f(U_1)=0$.
    \item [(iii)] $f(W) \subseteq U_1 \oplus W^-$.
    \item [(iv)]$g(W^-) \subseteq U_1 \oplus W^-$.
\end{enumerate}
Indeed claims (i), (ii), and (iii) follows directly from the explicit definitions of $f$ and $J$ (see \eqref{eq def f}, \eqref{eq def J}) and these subspaces. For claim (iv), since ${}^t gJ= -J g$, claim (i) implies that
\[{}^tg (U_2)= {}^t g J (U_1)= J g (U_1)\subseteq  J(U_1)=U_2. \]
This implies claim (iv).

Now combining claims (ii), (iv), we obtain
\begin{align}\label{eq f- g- 1}
    f \circ \iota \circ p \circ g \circ \iota= f  \circ g \circ \iota.
\end{align}
Also, the condition that $g(U_1)\subseteq U_1$ implies that  
\[p \circ g \circ \iota \circ p(U_1)=p \circ g(U_1)=0.\]
Hence, by claim (iii), we obtain that
\begin{align}\label{eq f- g- 2}
   p \circ g \circ \iota \circ p \circ f \circ \iota=  p \circ g \circ f \circ \iota.
\end{align}
Therefore,
    \[ f_{W^-} \circ g_{W^-}= p \circ f \circ \iota \circ p \circ g \circ \iota 
    \overset{\eqref{eq f- g- 1}}{=} p \circ f \circ g \circ \iota = p \circ g \circ f \circ \iota \overset{\eqref{eq f- g- 2}}{=} p \circ g \circ \iota \circ p \circ f \circ \iota=g_{W^-} \circ f_{W^-}. \]
\end{proof}

\begin{lemma}\label{lem second conjugation}
    Let $g \in C_{\widehat{\phi}}^{\ast} \cap C(f)$ be an element constructed in Lemma \ref{lem first conjugation}. There exists an $h \in H_{\lambda}$ such that $g':=h^{-1} g h$ satisfying that $g'(U)\subseteq U$, $g'(W^-) \subseteq W^-$, and $(g')_{W^-}=g_{W^-}$. 
    Note that $g'$ may not lie in $C(f)$.
\end{lemma}
\begin{proof}
    By the proof of \cite[Proposition II.6]{MW86}, there exists an invertible endomorphism $h_{U_1\oplus W^-}\in \End(U_1 \oplus W^-)^{\deg =0}$ of the form
    \[ h_{U_1\oplus W^-}(\Delta_i(a))= \begin{cases}
        \Delta_i(a) &\text{ if }a\leq d-r,\\
        \Delta_i(a) +\lambda(i,a)\Delta_{d-a}(a) &\text{ if }a>d-r,
    \end{cases}\]
    where $\lambda(i,a)$ are some complex numbers, with the following property.
    Define 
    \[(g')_{U_1\oplus W^-}:=h_{U_1\oplus W^-}^{-1} \circ g_{U_1 \oplus W^-} \circ h_{U_1\oplus W^-}.\]
    Then
    \begin{itemize}
        \item $(g')_{U_1\oplus W^-}(U_1)\subseteq U_1$ and its restriction to $U_1$ is equal to $g_{U_1}$, and
        \item $(g')_{U_1\oplus W^-}(W^-)\subseteq W^-$ and its restriction to $W^-$ is equal to $g_{W^-}$.
    \end{itemize}
    Now we show that there exists a (not necessarily unique) extension $h\in H_{\lambda}$ of $h_{U_1\oplus W^-}$. Indeed, the equation ${}^th J h=J$ is equivalent to
    two equations 
    \begin{align*}
        \begin{cases}
            J_{W^{-}} h_{U_2,W^-}+ {}^t h_{W^-,U_1} J_{U_2,U_1}=0,\\
            J_{U_1,U_2} h_{U_2,U_1}+ {}^t h_{U_2,W^-} J_{W^-} h_{U_2,W^-} + {}^t h_{U_2,U_1} J_{U_2,U_1}=0.
        \end{cases}
    \end{align*}
Since $h_{W^-,U_1}$ is given and $J_{W^-}, J_{U_2,U_1}$ are invertible, the first equation uniquely determines $h_{U_2,W^-}$. Next, recall that ${}^t J= \epsilon J =J^{-1}$, which implies that ${}^tJ_{W^-}=\epsilon J_{W^-}$ and ${}^t J_{U_{2},U_1}^{-1}= \epsilon J_{U_1,U_2}$. Therefore, if we take
\[ h_{U_2,U_1}= \frac{- J_{U_2,U_1}^{-1} {}^t h_{U_2,W^-}^{} J_{W^-}^{} h_{U_2,W^-}^{}}{2},\]
then the second equation holds. This establishes the existence of an extension $h$.

Finally, it is a straightforward computation that $g':= h^{-1} g h$ sends $U_1$ to $U_1$, $W^-$ to $W^-$, and $(g')_{W^-}= g_{W^-}$ by the design of $h$. It remains to show that $g'(U_2) \subseteq U= U_1 \oplus U_2$. Since $J(W^-)=W^-$ and ${}^tg' J+Jg'=0$, we have
\[{}^t g'(W^-)= {}^t g' J(W^-)= J g'(W^-) \subseteq W^-.\]
Therefore, $(g')_{U_2,W^-}=0$. This completes the proof of the lemma.
    \end{proof}

\subsection{Proof of Theorem \ref{thm bad}}
With the preparation so far, we prove Theorem \ref{thm bad} now.

We apply the induction on rank. In particular, we assume that $\widehat{\phi^{\#}}= \AZ_{bad}(\phi^{\#})$. According to Lemmas \ref{lem first conjugation}, \ref{lem g-} and \ref{lem second conjugation}, there exists a $g' \in C_{\widehat{\phi}}^{\ast}$ such that 
\begin{itemize}
    \item [(1)] $g'(U)\subseteq U$ and $g'(W^-) \subseteq W^-$, and
    \item [(2)] $(g')_{W^-} \in C(f_{W^-})\cap V_{\lambda^{\#}}^{\ast}$.
\end{itemize}
By (1), we have $g'= (g')_{U} \oplus (g')_{W^-}$. Therefore, we have the following equation for the rank matrix 
\begin{align}\label{eq bad 0}
   r(\widehat{\phi})=r(g')= r( \iota_{U,W}\circ (g')_{U}\circ p_{W,U}) + r(\iota_{W^-,W}\circ (g')_{W^-}\circ p_{W,W^-}). 
\end{align}
Here we compose the inclusion and projection just for matching the size of the rank matrices.

On the other hand, recall that $f_{W^-}\in C_{\phi^{\#}}$. Thus (2) and the induction hypothesis imply that 
\begin{align}\label{eq bad 1}
     r((g')_{W}) \leq r( \widehat{\phi^{\#}})= r( \AZ(\phi^{\#})).
\end{align}
The same holds after composing with $\iota_{W^-,W}$ and $p_{W^-,W}$.

Next, define 
\[ \lambda':= \bigoplus_{i=0}^r (\lvert\cdot\rvert^{d-r} \oplus \lvert \cdot \rvert^{r-d}). \]
Then since $ (g')_{U}$ preserves the $\epsilon$-symmetric bilinear $J_{U}$, we see that $(g')_{U}$ lies in $V_{\lambda'}^{\ast}$. Example \ref{ex bad open} then implies that
\begin{align}\label{eq bad 2}
    r( (g')_{U}) \leq r( \phi_{[d-r,d],[-d,-d+r]} )
\end{align}
since $\phi_{[d-r,d],[-d,-d+r]}$ corresponds to the unique open orbit in $V_{\lambda'}$. The same holds after composing with $\iota_{U,W}$ and $p_{U,W}$. 

Now \eqref{eq bad 0}, \eqref{eq bad 1}, \eqref{eq bad 2} and Theorem \ref{thm LM25} imply that
\[r(\widehat{\phi}) \leq r( \phi_{\{[d-r,d],[-d,-d+r]\}} \oplus \AZ(\phi^{\#}))= r(\AZ(\phi)).  \]
Since $\phi \mapsto \widehat{\phi}$ and $\phi \mapsto \AZ(\phi)$ are both involutions on $\Phi_{\lambda}(G)$, we see that $\widehat{\phi}= \AZ(\phi)$ by Lemma \ref{lem involution poset}. This completes the proof of the theorem. \qed

\section{Pyasetskii involution and ABV-packets}\label{sec ABV}
In this section, we recall the background of ABV-packets and a conjecture that the Aubert-Zelevinsky involution sends an ABV-packet $\Pi_{\phi}^{\ABV}$ to the ABV-packet associated with its Pyasetskii involution $\Pi_{\widehat{\phi}}^{\ABV}$. Then we show that Theorem \ref{thm bad} is evidence of this conjecture.

In \cite{Art89}, Arthur formulated a series of influential conjectures concerning global and local Arthur packets. Let $F$ be a non-Archimedean field of characteristic 0. A local Arthur parameter $\psi$ for $G = \RG(F)$ is defined as a $G^{\vee}$-conjugacy class of admissible homomorphisms
\[
\psi: W_F \times \SL_2(\BC) \times \SL_2(\BC) \to {}^L G
\]
such that the restriction of $\psi$ to $W_F$ has bounded image. Associated to each such $\psi$, there is an $L$-parameter $\phi_{\psi}$ given by
\[
\phi_{\psi}(w, x, y) := \psi\left(w, x, \begin{pmatrix} |w|^{1/2} & \\ & |w|^{-1/2} \end{pmatrix}\right).
\]
Arthur conjectured that for every local Arthur parameter $\psi$ of $G$, one can attach a finite (multi-)set $\Pi_{\psi}$ of irreducible smooth representations of $G$, called the \emph{local Arthur packet}, which satisfies certain (twisted) endoscopic character relations. The members of each local Arthur packet are expected to arise as local components of square-integrable automorphic representations in global Arthur packets. Consequently, the local Arthur packets are expected to consist entirely of unitary representations. 

For real groups, Adams-Barbasch-Vogan (\cite{ABV92}) introduced a geometric framework for constructing local Arthur packets. Specifically, for each $L$-parameter $\phi$, they defined the ABV-packet (also known as the micro-packet) $\Pi_{\phi}^{\ABV}$, and conjectured that whenever the local Arthur packet $\Pi_{\psi}$ exists, it coincides with the corresponding ABV-packet: $\Pi_{\psi} = \Pi_{\phi_{\psi}}^{\ABV}$. This conjecture was subsequently established by Adams-Arancibia-Mezo (\cite{AAM24}) for a broad family of classical groups, in the cases where local Arthur packets are defined. Thus, the ABV-packets may be viewed as a generalization of local Arthur packets.

Over $p$-adic fields, Vogan extended the notion of ABV-packets in \cite{Vog93}, assuming the existence of an enhanced local Langlands correspondence. This definition was subsequently reformulated by Cunningham, Fiori, Moussaoui, Mracek, and Xu in \cite{CFMMX22}, and it is expected that these two approaches agree. For each $L$-parameter $\phi$, we consider the ABV-packet $\Pi_{\phi}^{\ABV}$ as defined in \cite[\S8]{CFMMX22}. We recall certain basic properties of $\Pi_{\phi}^{\ABV}$ from the geometric construction.

\begin{prop}[{\cite[Theorem 7.22(b),(d)]{CFMMX22}}]\label{prop ABV}
    Let $\phi \in \Phi_{\lambda}(G)$.
    \begin{enumerate}
        \item The ABV-packet $ \Pi_{\phi}^{\ABV}$ contains the $L$-packet $\Pi_{\phi}$.
        \item For any $\pi \in \Pi_{\phi}^{\ABV}$, we have $\phi_{\pi}\geq_C \phi$, where $\phi_{\pi}$ is the $L$-parameter of $\pi$ under the local Langlands correspondence.
    \end{enumerate}
\end{prop}

As in the real case, it is conjectured that local Arthur packets are special cases of ABV-packets. In line with \cite[\S8]{CFMMX22}, we refer to this expectation as Vogan's conjecture:

\begin{conj}\label{conj Vogan}
    Assume that the enhanced local Langlands correspondence for $G$ is established, as well as the theory of local Arthur packets. Then, for any local Arthur parameter $\psi$ of $G$, we have $\Pi_{\phi_{\psi}}^{\ABV}= \Pi_{\psi}$.
\end{conj}
Note that there is a more refined version of this statement asserting that the various distributions associated to local Arthur packets and ABV-packets also coincide; see \cite[Conjecture 8.3.1]{CFMMX22}. Vogan's conjecture for $p$-adic groups is proved for $G=\GL_n(F)$ by Cunningham-Ray in \cite{CR22, CR26}, and independently by Riddlesden (\cite{Rid23}) and the second name author (\cite{Lo24}). In an ongoing joint work with Cunningham-Liu-Ray-Xu(\cite{CHLLRX}), we will prove the refined version of Vogan's conjecture for quasi-split symplectic, orthogonal, and unitary groups.

A fundamental property of local Arthur packets for classical groups is their compatibility with the Aubert-Zelevinsky involution. Specifically, given a local Arthur parameter $\psi$, we define its dual $\widehat{\psi}$, by
\[
\widehat{\psi}(w, x, y) = \psi(w, y, x) \qquad \text{for all } w \in W_F,\ x, y \in \SL_2(\BC).
\]
This notation is compatible with the Pyasetskii involution in the sense that $\phi_{\widehat{\psi}} = \widehat{\phi_{\psi}}$ (\cite[Corollary 6.10]{CFMMX22}). For classical groups, it is known (see \cite[Lemma 4.4.1]{AGIKMS24}) that
\[
\Pi_{\widehat{\psi}} = \{ \widehat{\pi}\mid \pi \in \Pi_{\psi} \}.
\]
Motivated by Vogan's Conjecture, it is expected that an analogous compatibility holds for ABV-packets not of Arthur type as well. Furthermore, the Aubert-Zelevinsky involution is believed to be related to the Fourier transform, which itself is compatible with the construction of ABV-packets. See \cite[\S 10.2.8, 10.3.4]{CFMMX22}.

\begin{conj}[{\cite[Section 10.3.4]{CFMMX22}}]\label{conj AZ}
    For any $\phi \in \Phi_{\lambda}(G)$, we have 
    \[\Pi_{\widehat{\phi}}^{\ABV}=\{\widehat{\pi}\ | \ \pi \in \Pi_{\phi}^{\ABV}\}.\]
\end{conj}

We remark that $\Pi_{\widehat{\phi}}(G)$ and the set $\widehat{\Pi}_{\phi}(G) := \{\widehat{\pi} \mid \pi \in \Pi_{\phi}(G)\}$ need not coincide. When $G$ is a classical group and $\phi \in \Phi(G)$ is isotypic of bad parity, the $L$-packet $\Pi_{\phi'}(G)$ is a singleton for every $\phi' \in \Phi_{\lambda_{\phi}}(G)$ (Remark \ref{rmk gp}(2)).  In this situation, we show in the Proposition \ref{prop ABV AZ bad} below that Conjecture \ref{conj AZ} implies that $\Pi_{\widehat{\phi}}(G)= \widehat{\Pi}_{\phi}(G)$. This is indeed our motivation for Theorem \ref{thm bad}. In other words, Theorem \ref{thm bad} provides evidence for Conjecture \ref{conj AZ}.

\begin{prop}\label{prop ABV AZ bad}
Let $G$ be a connected reductive algebraic group defined and quasisplit over $F$.  Assume Conjecture \ref{conj AZ} holds for $G$. Suppose that $\lambda$ is an infinitesimal parameter of $G$ such that the $L$-packets $\Pi_{\phi}(G)$ are singletons for all $\phi \in \Phi_{\lambda}(G)$. Then
\[ \Pi_{\widehat{\phi}}(G)=\{ \widehat{\pi}\ | \ \pi \in \Pi_{\phi}(G)\}. \]
\end{prop}
\begin{proof}
For each $L$-parameter $\phi \in \Phi_{\lambda}(G)$, let $\pi_{\phi}$ denote the unique member of $\Pi_{\phi}(G)$ and let $\widetilde{\phi}$ be the $L$-parameter of $\widehat{\pi_{\phi}}$. Since the Aubert-Zelevinsky involution is an involution, the map $\phi \mapsto \widetilde{\phi}$ is an involution on $\Phi_{\lambda}(G)$.

By Proposition \ref{prop ABV}(1), $\pi_{\phi} \in \Pi_{\phi}^{\ABV}$. Conjecture \ref{conj AZ} then gives $\widehat{\pi_{\phi}} \in \Pi_{\widehat{\phi}}^{\ABV}$, so Proposition \ref{prop ABV}(2) implies $\widetilde{\phi} \geq_C \widehat{\phi}$ for all $\phi \in \Phi_{\lambda}(G)$. Lemma \ref{lem involution poset} now yields $\widehat{\phi} = \widetilde{\phi}$ for all $\phi \in \Phi_{\lambda}(G)$. This completes the proof of the proposition.
\end{proof}

\end{document}